\documentclass[a4paper,11pt,fullpage]{article}
\usepackage[english]{babel}

\usepackage[left=1in,right=1in,top=1in,bottom=1in]{geometry}

\usepackage[usenames]{color}
\usepackage{colortbl}

\usepackage{indentfirst}
\usepackage{amsmath}
\usepackage{amsthm}
\usepackage{amsfonts}
\usepackage{amssymb}
\usepackage{amsxtra}
\usepackage{latexsym}
\usepackage{ifthen}
\usepackage{cite}
\usepackage{esint}
\usepackage[cp1250]{inputenc}

\usepackage{epic}
\usepackage{eepic}

\DeclareMathOperator*{\esssup}{ess\,sup}

\DeclareMathOperator*{\essinf}{ess\,inf}

\DeclareMathOperator*{\osc}{osc}

\numberwithin{equation}{section}
\newtheorem{theorem}{Theorem}[section]
\newtheorem{lemma}{Lemma}[section]
\newtheorem{remark}{Remark}[section]
\newtheorem{definition}{Definition}[section]

\def\XXint#1#2#3{{\setbox0=\hbox{$#1{#2#3}{\int}$}
     \vcenter{\hbox{$#2#3$}}\kern-.5\wd0}}

\begin{document}

\title{On asymptotic behavior of solutions to non-uniformly elliptic equations with generalized Orlicz growth
}

\author{ O.V. Hadzhy, M.O. Savchenko, I.I. Skrypnik, M.V. Voitovych
 }

  \maketitle

  \begin{abstract}
We study asymptotic behavior of sub-solutions to non-uniformly elliptic equations with nonstandard growth. In particular,
 Harnack type inequalities  are proved. Our approach gives new results for the cases with $(p,q)$ nonlinearity
and generalized Orlicz growth.



\end{abstract}

\pagestyle{myheadings} \thispagestyle{plain}
\markboth{}
{On asymptotic behavior ....}

\section{Introduction and main results}\label{Introduction}

To explain the point of view of this research  consider the following equations
$$ div\bigg(H_{i}(x, |\nabla u|)\frac{\nabla u}{| \nabla u|^{2}}\bigg)=0,\quad i=1,2,$$
$$H_{1}(x, v)=v^{p}+ a_{1}(x) v^{q},\quad a_{1}(x)\geqslant 0, \quad v>0,$$
and
$$H_{2}(x,v)=v^{p}\big(1+ a_{2}(x)\log (1+v) \big),\quad a_{2}(x)\geqslant 0, \quad v>0.$$
It is well known (see \cite{BarColMing})  that the correspondent non-negative bounded local weak solutions satisfy
Harnack's type inequality if
$$
\osc\limits_{B_{r}(x_{0})}a_{1}(x)\leqslant A_{1}\,r^{q-p} \quad \text{and}\quad
\osc\limits_{B_{r}(x_{0})}a_{2}(x)\leqslant A_{2}\left[\log\frac{1}{r}\right]^{-1},\quad A_{1},A_{2} > 0,
$$
or more generally (see \cite{HadSkrVoi}), Harnack's type inequality holds  under the conditions
$$
\osc\limits_{B_{r}(x_{0})}a_{1}(x)\leqslant A_{1}\,\left[\log\frac{1}{r}\right]^{L_{1}}\,r^{q-p}\quad \text{and}\quad
\osc\limits_{B_{r}(x_{0})}a_{2}(x)\leqslant A_{2}\,\left[\log\log\frac{1}{r}\right]^{L_{2}}\left[\log\frac{1}{r}\right]^{-1},\quad A_{1},A_{2} > 0,
$$
if $L_{1}, L_{2} >0$ are sufficiently small.

Now let $a_{1}(x)= a_{2}(x)= \big|\log|\log\frac{1}{|x-x_{0}|}|\big|^{L}, L\in \mathbb{R}^{1}$ and set $G_{1}(v)=v^{p}+  v^{q}$,\\ $G_{2}(v)=v^{p}\big(1+ \log (1+v) \big)$,
then  evidently we have
$$\gamma^{-1}\,\left|\log\left|\log\frac{1}{|x-x_{0}|}\right|\right|^{L}\,G_{i}(v) \leqslant H_{i}(x, v) \leqslant \gamma\,G_{i}(v),\quad i=1,2,\quad \text{if}\quad L<0,$$
and
$$\gamma^{-1}\,G_{i}(v) \leqslant H_{i}(x, v) \leqslant \gamma\,\left|\log\left|\log\frac{1}{|x-x_{0}|}\right|\right|^{L}\,G_{i}(v),\quad i=1,2,\quad \text{if}\quad L>0,$$
for $v>0$ and for $x\in B_{R}(x_{0})$, if $R$ is sufficiently small.

By our main Theorems (see results below) Harnack's type inequality for
non-negative bounded local weak solutions of the correspondent equation will be valid if $L$ is
sufficiently small. One can see that in this sense our results improve the previous ones. However
our point of view is to take under consideration also non uniformly elliptic equations of the
previous type. Of course, here it would be interesting to unify our approach.

More precisely, in this paper we are concerned with elliptic equations
\begin{equation}\label{eq1.1}
 {\rm div}\bigg( H(x, | \nabla u|) \frac{\nabla u}{| \nabla u|^{2}} \bigg) = 0,\quad H(x,v):= \int\limits_{0}^{v}h(x,s)\,ds,\quad v>0,\quad x\in\Omega,
\end{equation}
where $\Omega$ is bounded domain in $\mathbb{R}^{n}$, $n\geqslant2$. We suppose that the function $ h(x, v):\Omega\times\mathbb{R}^{1}_{+}\rightarrow\mathbb{R}^{1}_{+}$ is such that $h (\cdot, v)$ is Lebesgue measurable for all $v\in \mathbb{R}^{1}_{+}$,
and $h(x,\cdot)$ is increasing and continuous for almost all $x\in\Omega$, $\lim\limits_{v\rightarrow 0}h(x, v)=0,
\lim\limits_{v\rightarrow \infty} h(x, v)=\infty$ .

We assume also that the following structure conditions are satisfied
\begin{equation}\label{eq1.2}
K_{1} \,a(x)\,g(x, v) \leqslant h(x, v) \leqslant K_{2} \,b(x)\,g(x, v),\quad x\in\Omega, v\in \mathbb{R}^{1}_{+},
\end{equation}
where $K_{1}$, $K_{2}$ are positive constants.

This type of equations belongs to a wide class of non uniformly elliptic equations with generalized
Orlicz growth. In terms of the function $g$, this class can be characterized as follows.
\begin{itemize}
\item[($g$)]
There exist $1<p<q$ such that  for $ w\geqslant v> 0$ there holds
\begin{equation*}\label{gqineq}
\left( \frac{ w}{ v} \right)^{p-1} \leqslant\frac{g(x, w)}{g(x, v)}\leqslant
 \left( \frac{w}{v} \right)^{q-1},\quad x\in \Omega.
\end{equation*}
\end{itemize}
We also assume that one of the following conditions holds:
\begin{itemize}
\item[ ($g_{\mu}$)]
there exists $R_{0} >0$ and  non-increasing  function $\mu(r)\geqslant 1$ for $r\in (0, R_{0}),$ $\lim\limits_{r\rightarrow 0}\,\mu(r)=+\infty, $ $\lim\limits_{r\rightarrow 0}\,\mu(r) r=0,$ such that for any $K>0$ there holds
\begin{equation*}
g\left(x, \frac{v}{r}\right) \leqslant C(K)\,\mu(r) g\left(y, \frac{v}{r}\right),\,\,\, \text{for any}\,\, x,y\in B_{r}(x_{0})\subset B_{R_{0}}(x_{0})\subset \Omega \,\,\, \text{and} \,\,\, r\leqslant v \leqslant K,
\end{equation*}
with some $C(K)>0$,
 \end{itemize}
or
\begin{itemize}
\item[ ($ g_{\lambda}$)]
there exists $R_{0} >0$ and  non-decreasing  function $0< \lambda(r)\leqslant 1$ for $r\in (0, R_{0}),$\\ $\lim\limits_{r\rightarrow 0}\lambda(r)= 0, $ $\lim\limits_{r\rightarrow 0}\,\dfrac{\lambda(r)}{r}= +\infty,$ such that for any $K>0$ there holds
\begin{equation*}
g\left(x, \frac{v}{r}\right) \leqslant C(K)\, g\left(y, \frac{v}{r}\right),\,\,\, \text{for any}\,\, x,y\in B_{r}(x_{0})\subset B_{R_{0}}(x_{0})\subset \Omega \,\,\, \text{and} \,\,\, r\leqslant v \leqslant \lambda(r)\,K,
\end{equation*}
with some $C(K)>0$.
 \end{itemize}

For the functions $a(x), b(x) \geqslant 0$ we assume that
\begin{equation}\label{eq1.3}
a^{-1}(x)\in L^{t}_{loc}(\Omega),\quad b(x) \in L^{s}_{loc}(\Omega),\quad \frac{1}{tp} +\frac{1}{sq} +\frac{1}{p} -\frac{1}{q} <\frac{1}{n},
\end{equation}
$$t\in\left(\max\left(1,\frac{q-p+1}{p-1}\right),\infty\right],\quad s\in(1,\infty].$$
\begin{remark}\label{rem1.1}
{\rm
We note that the assumptions on the functions $a(x)$ and $b(x)$ imply the local boundedness of  sub-solutions to Eq. \eqref{eq1.1}, for this we
refer the reader to the paper by Cupini, Marcellini and Mascolo \cite{CupMarMas}.
}
\end{remark}
Here some typical examples of the functions $g(x, v)$ and $h(x, v)$, which satisfy the above conditions:
\begin{itemize}
\item $g(x, v)= v^{p(x)-1},\quad \osc\limits_{B_{r}(x_{0})} p(x) \leqslant A\,\dfrac{\bar{\mu}(r)}{\log\frac{1}{r}},\,\quad
\lim\limits_{r\rightarrow 0}\bar{\mu}(r)=\infty, \quad \lim\limits_{r\rightarrow 0}\dfrac{\bar{\mu}(r)}{\log\frac{1}{r}}=0$ \\satisfies condition
($g_{\mu}$) with $\mu(r)=\exp\big(A\bar{\mu}(r)\big)$,

\item $g(x, v)= v^{p-1}+ a_{1}(x) v^{q-1}$,\,\, $\osc\limits_{B_{r}(x_{0})}a_{1}(x)\leqslant A_{1}\,\bar{\mu}(r)\,r^{q-p},\,\,
\lim\limits_{r\rightarrow 0}\bar{\mu}(r)=\infty,\,\, \lim\limits_{r\rightarrow 0}\bar{\mu}(r)\,r^{q-p}=0,$\\
satisfies conditions ($g_{\mu}$), ($g_{\lambda}$) with $\mu(r)=\bar{\mu}(r)$ and $\lambda(r)=[\bar{\mu}(r)]^{-\frac{1}{q-p}}$,

\item $g(x, v)=v^{p-1}\big(1+ \log(1+a_{2}(x)\,v)\big)$, $\osc\limits_{B_{r}(x_{0})}a_{2}(x) \leqslant A_{2}\,\bar{\mu}(r)\,r$,
 $\lim\limits_{r\rightarrow 0}\bar{\mu}(r)=\infty, \lim\limits_{r\rightarrow 0}\bar{\mu}(r)\,r=0,$\\
satisfies conditions ($g_{\mu}$), $(g_{\lambda})$ with  $\mu(r)=\log\big(A_{2}\bar{\mu}(r)\big)$ and $\lambda(r)=[\bar{\mu}(r)]^{-1}$,

\item $h(x,v)= a(x)\,v^{p(x)-1}$, $h(x,v)= v^{p-1} +a(x) v^{q-1}$, $h(x,v)=v^{p-1}\big(1+a(x)\log(1+a_{2}(x)\,v)\big)$,
$a^{-1}(x)\in L^{t}_{loc}(\Omega),\quad a(x)\in L^{s}_{loc}(\Omega)$ satisfy conditions \eqref{eq1.2}, \eqref{eq1.3} and $(g)$.
\end{itemize}

Before formulating the main results, let us recall the definition of a  weak solution
to Eq. \eqref{eq1.1}. We write $W(\Omega)$ for the class of functions
$u\in W^{1,1}(\Omega)$ with $\int\limits_{\Omega}H (x,|\nabla u|)\,dx<+\infty$.
We also need a class of functions $W_{0}(\Omega )$ which consists of functions
$u\in W_{0}^{1,1}(\Omega)$
such that $\int\limits_{\Omega}H(x,|\nabla u|)\,dx<+\infty$.
\begin{definition}
{\rm
We say that a function $u$ is a  weak sub(super)-solution to Eq. \eqref{eq1.1} if \\$u\in W(\Omega)$ and the integral identity
\begin{equation}\label{eq1.4}
\int\limits_{\Omega} H(x, |\nabla u|)\,\frac{\nabla u}{|\nabla u|^{2} }\nabla\varphi \,dx\leqslant(\geqslant)\,0,
\end{equation}
holds for all non-negative test functions $\varphi\in W_{0}(\Omega)$.
}
\end{definition}

Let  $a_{1}(x)=1+a^{-1}(x),\quad b_{1}(x)=1+b(x)$, for any $x_{0}\in \mathbb{R}^{n}$ and $r >0$ set
\begin{multline*}
 \Lambda(x_{0}, r) :=  \bigg(r^{-n}\,\int\limits_{B_{r}(x_{0})} a^{t}_{1}(x) dx\bigg)^{\frac{1}{tp}} \bigg(r^{-n}\,\int\limits_{B_{r}(x_{0})} b^{s}_{1}(x) dx \bigg)^{\frac{1}{sp}}+ \\ + \bigg(r^{-n}\,\int\limits_{B_{r}(x_{0})} a^{t}_{1}(x) dx\bigg)^{\frac{q-1}{tp(p-1)}} \bigg(r^{-n}\,\int\limits_{B_{r}(x_{0})} b^{s}_{1}(x) dx \bigg)^{\frac{1}{sq}}.
\end{multline*}
We refer to the parameters $n$, $p$, $q$, $t$, $s$, $K_{1}, K_{2}$, $M:=\sup\limits_{\Omega}\,\, u$ and $C(M)$    as our structural
data, and we write $\gamma$ if it can be quantitatively determined as a priory in terms of the above
quantities. The generic constant $\gamma$ may change from line to line.

Our first main result of this paper reads as follows.
\begin{theorem}\label{th1.1}
Let $u$ be a   non-negative weak super-solution to Eq. \eqref{eq1.1}, let conditions \eqref{eq1.2}, \eqref{eq1.3}, ($g$) and ($g_{\lambda}$) be fulfilled, then there exist positive constants $\gamma $ and $\beta_{1}$ depending only on the data, such that for any $m >0$ and any $B_{8\rho}(x_{0})\subset \Omega$ there holds
\begin{equation}\label{eq1.5}
m\,\lambda(\rho)\,\bigg(\frac{|E(\rho, m)|}{\rho^{n}}\bigg)^{\frac{t+1}{t(p-1)}}  \leqslant \gamma\,\exp\left(\gamma\,\left[\Lambda\left(x_{0},\frac{\rho}{4}\right)\right]^{\beta_{1}}\right)\left(\inf\limits_{B_{\frac{\rho}{2}}(x_{0})}u(x) + \rho \right),
\end{equation}
where $ E(\rho, m):= B_{\rho}(x_{0})\cap \{u(x) > m \}$.
\end{theorem}
\begin{theorem}\label{th1.2}
Let $u$ be a  non-negative weak super-solution to Eq. \eqref{eq1.1}, let conditions \eqref{eq1.2}, \eqref{eq1.3}, ($g$) and ($g_{\mu}$) be fulfilled, then there exist positive constants $\gamma $ and $\beta_{1}$ depending only on the data, such that for any $m>0$ and any $B_{8\rho}(x_{0})\subset \Omega$ there holds
\begin{equation}\label{eq1.6}
m\,\bigg(\frac{|E(\rho, m)|}{\rho^{n}}\bigg)^{\frac{t+1}{t(p-1)}}  \leqslant \gamma\,\exp\left(\gamma\,\left[\mu\left(\frac{\rho}{4}\right)\,\Lambda\left(x_{0},\frac{\rho}{4}\right)\right]^{\beta_{1}}\right)\left(\inf\limits_{B_{\frac{\rho}{2}}(x_{0})}u(x) + \rho \right).
\end{equation}
\end{theorem}

The following two results are consequences of Theorems \ref{th1.1}, \ref{th1.2}. The first one is the weak Harnack type inequality.
\begin{theorem}\label{th1.3}
Let $u$ be a  weak non-negative super-solution to Eq.\eqref{eq1.1}, let conditions of Theorem \ref{th1.1} be fulfilled. Then for all $\theta \in\left(0, \min\left(1,\frac{t}{t+1}(p-1)\right)\right)$ there holds
\begin{equation*}
\left(\rho^{-n}\int\limits_{B_{\rho}(x_{0})} u^{\theta}(x)\,dx\right)^{\frac{1}{\theta}}
\leqslant \left(t(p-1)-\theta(t+1)\right)^{-\frac{1}{\theta}}\frac{\gamma}{\lambda(\rho)}\times
\qquad\qquad\qquad\qquad\qquad
\qquad\qquad\qquad\qquad
\end{equation*}
\begin{equation}\label{eq1.7}
\times\exp\left(\gamma\,\left[\Lambda\left(x_{0},\frac{\rho}{4}\right)\right]^{\beta_{1}}\right)
\left(\inf\limits_{B_{\frac{\rho}{2}}(x_{0})}u(x) + \rho \right),
\end{equation}
provided that $B_{8\rho}(x_{0})\subset \Omega$.

Likewise, let $u$ be a  weak non-negative super-solution to Eq.\eqref{eq1.1} and let conditions of Theorem \ref{th1.2} be fulfilled. Then for all $\theta \in\left(0, \min\left(1,\frac{t}{t+1}(p-1)\right)\right)$ there holds
\begin{equation*}
\bigg(\rho^{-n}\,\int\limits_{B_{\rho}(x_{0})} u^{\theta}(x)\,dx\bigg)^{\frac{1}{\theta}}
\leqslant \big(t(p-1)-\theta(t+1)\big)^{-\frac{1}{\theta}}\times\qquad\qquad\qquad\qquad\qquad
\qquad\qquad\qquad\qquad
\end{equation*}
\begin{equation}\label{eq1.8}
\times\gamma\exp\left(\gamma\,\left[\mu\left(\frac{\rho}{4}\right)
\Lambda\left(x_{0},\frac{\rho}{4}\right)\right]^{\beta_{1}}\right)\left(\inf\limits_{B_{\frac{\rho}{2}}(x_{0})}u(x) + \rho \right),
\end{equation}
provided that $B_{8\rho}(x_{0})\subset \Omega$. Here $\beta_{1} > 0$ is the number defined in Theorems \ref{th1.1}, \ref{th1.2} and $\gamma >0$ is a constant, depending only on the data.
\end{theorem}
Next result is the Harnack inequality.
\begin{theorem}\label{th1.4}Let $u$ be a   non-negative weak solution to Eq. \eqref{eq1.1}, let conditions \eqref{eq1.2}, \eqref{eq1.3}, ($g$),  ($g_{\lambda}$) and ($g_{\mu}$) be fulfilled, then there exist positive constants $\gamma $, $\beta_{2}$, $\beta_{3}$ depending only on the data, such that
\begin{equation}\label{eq1.9}
\sup\limits_{B_{\frac{\rho}{2}}(x_{0})}u(x)\leqslant \frac{\gamma}{\lambda(\rho)}\,\left[\mu\left(\frac{\rho}{4}\right)\right]^{\beta_{2}}\,
\exp\left(\gamma\,\left[\Lambda\left(x_{0},\frac{\rho}{4}\right)\right]^{\beta_{3}}\right)\left(\inf\limits_{B_{\frac{\rho}{2}}(x_{0})}u(x) + \rho \right),
\end{equation}
provided that $B_{8\rho}(x_{0}) \subset \Omega$.

Likewise, let $u$ be a   non-negative weak solution to Eq. \eqref{eq1.1}, let conditions \eqref{eq1.2}, \eqref{eq1.3}, ($g$) and ($g_{\mu}$) be fulfilled, then
\begin{equation}\label{eq1.10}
\sup\limits_{B_{\frac{\rho}{2}}(x_{0})}u(x)\leqslant \gamma\,\exp\left(\gamma\,\left[\mu\left(\frac{\rho}{4}\right)\Lambda\left(x_{0},\frac{\rho}{4}\right)\right]^{\beta_{4}}\right)
\left(\inf\limits_{B_{\frac{\rho}{2}}(x_{0})}u(x) + \rho \right),
\end{equation}
provided that $B_{8\rho}(x_{0}) \subset \Omega$.
\end{theorem}

Before describing the method of proof, a few words about the history of the problem. Qualitative properties of solutions to the corresponding elliptic equations in the standard case, i.e. if $p=q$ are well known since
the famous results by De Giorgi \cite{DeGiorgi}, Nash \cite{Nash}  and Moser \cite{Moser} (we refer the reader to the well-known monograph of Ladyzhenskaya and Ural'tseva  \cite{LadUr}, and  the seminal papers of Serrin \cite{Serrin},  DiBenedetto and Trudinger \cite{DiBTru}). Harnack's inequality for non uniformly elliptic equations has been known since the well-known paper of Trudinger \cite{Tru}.

The study of regularity of minima of functionals with non-standard growth has been initiated by Zhikov
\cite{ZhikIzv1983, ZhikIzv1986, ZhikJMathPh94, ZhikJMathPh9798, ZhikKozlOlein94},
Marcellini \cite{Marcellini1989, Marcellini1991}, and Lieberman \cite{Lieberman91},
and in the last thirty years, the qualitative theory
of second order elliptic equations with so-called log-condition
(i.e. if $\lambda(r) \equiv 1$ and $\mu(r) \equiv 1$) has been actively developed. Moreover, many authors have established local boundedness, Harnack's inequality  and continuity of solutions to such equations, as well as  local minimizers, $Q$-minimizers, and $\omega$-minimizers of the corresponding minimization problems (see, e.g. \cite{ Alhutov97, AlhutovMathSb05, AlhutovKrash04, AlhutovKrash08, AlkhSurnApplAn19, AlkhSurnJmathSci20, AlkhSurnAlgAn19, BarColMing, BarColMingStPt16, BarColMingCalc.Var.18, BenHarHasKarp20, BurchSkrPotAn, ColMing218, ColMing15, ColMingJFunctAn16, DienHarHastRuzVarEpn, Fan1995, FanZhao1999, HadSkrVoi, HarHastOrlicz, HarHastZAn19, HarHastLee18, HarHastToiv17, Krash2002, OkNA20, ShSkrVoit20, SkrVoitUMB19, SkrVoitNA20, SurnPrepr2018 } and references therein).

The case when conditions $(g_{\mu})$ or $(g_{\lambda})$ hold differs substantially from the logarithmic case.
To our knowledge, there are few results in this direction. Zhikov \cite{ZhikPOMI04} obtained a generalization
of the logarithmic condition which guarantees the density of smooth functions in Sobolev space
$W^{1,p(x)}(\Omega)$. Particularly, this result holds if $1< p < p(x)$ and
$$ |p(x)-p(y)|\leqslant \frac{|\log\mu(|x-y|)|}{|\log|x-y||},\quad x,y\in \Omega,\quad
\int\limits_{0}[\mu(r)]^{-\frac{n}{p}}\,\frac{dr}{r}=\infty.$$
Particularly, the function $\mu(r)=\left[\log\frac{1}{r}\right]^{L}$ satisfies the above condition if $L\leqslant \frac{p}{n}$.
Later Zhikov and Pastukhova \cite{ZhikPast2008MatSb} under the same condition proved higher integrability of the
gradient of solutions to the $p(x)$-Laplace equation.

Continuity and Harnack's inequality to the $p(x)$ -
Laplace equation were proved in \cite{AlhutovKrash08, AlkhSurnAlgAn19, SurnPrepr2018}  under the condition $(g_{\mu})$, where $\mu(r)$ satisfies
\begin{equation}\label{eq1.11}
\int\limits_{0}\,\exp\big(-c[\mu(r)]^{\beta}\big)\frac{dr}{r}=\infty,
\end{equation}
with some positive constants $c ,\beta$. Particularly, the function $\mu(r)= \left[\log\log\frac{1}{r}\right]^{L},\, L\,\beta <1$ satisfies the above condition. These results were generalized in \cite{ShSkrVoit20, SkrVoitUMB19, SkrVoitNA20} for a wide class of elliptic
equations with non-logarithmic Orlicz growth, furthermore, Harnack's inequality was proved
in \cite{ShSkrVoit20} under similar conditions. In the proof, the authors used Trudinger's ideas \cite{Tru}.
These results were refined in \cite{HadSkrVoi}, in addition, the Harnack inequality was proved under
the conditions $(g_{\lambda})$, $(g_{\mu})$ and
\begin{equation}\label{eq1.12}
\int\limits_{0}\lambda(r)[\mu(r)]^{-\beta}\frac{dr}{r}=\infty,
\end{equation}
with some $\beta >0$.

In this paper, as already mentioned, we improve previous results, in particular \eqref{eq1.11} and \eqref{eq1.12}. The main difficulty arising in the proof of the main results is related to the so-called theorem
on the expansion of positivity. Roughly speaking, having information on the measure of the
''positivity set'' of u over the ball $B_{r}(\bar{x})$:
$$
|\{x\in B_{r}(\bar{x}) : u(x) \geqslant m \}| \geqslant \alpha(r)\,|B_{r}(\bar{x})|,
$$
with some $r, m >0$ and $\alpha(r) \in(0,1)$, we cannot use the covering argument of Krylov and
Safonov \cite{KrlvSfnv1980}, as was done in the logarithmic case, i.e. if $\alpha$ is independent of $r$ (see e.g. \cite{BarColMing}). Moreover, we cannot use Moser's method, as was done in \cite{ShSkrVoit20, Tru} (it seems that it can be done under condition
$(g_{\mu})$, but not under conditions $(g_{\lambda})$ and $(g_{\mu})$). We
also cannot use the local clustering lemma \cite{DiBGiaVes}, since in this case we will inevitably arrive at an
estimate of the form
$$
\inf\limits_{B_{\rho}(\bar{x})} u \geqslant \gamma\,m\,\exp\left(-\gamma \int\limits_{\bar{r}}^{\rho} [\Lambda(\bar{x},s)]^{\beta}\,\frac{ds}{s}\right),
\quad \bar{r}=\varepsilon\,r\frac{\alpha^{2}(r)}{\mu(r)},
$$
with some positive $\gamma, \varepsilon$ and $\beta$, from which the weak Harnack type inequality, Theorem
\ref{th1.3} does not follows. At the end of the paper we will demonstrate the possibility of using the local clustering lemma, this can only be done under the conditions $(g_{\lambda})$ and   $\Lambda(x_{0}, \rho)\asymp$ const. Note that this proof is valid for the corresponding
$DG_{g,\Lambda}(\Omega)$ De Giorgi  classes (see Section \ref{Sec2} below).

In the present paper we use the workaround that goes back to Mazya \cite{Maz} and Landis papers \cite{Landis_uspehi1963, Landis_mngrph71}, due to which we obtain an analogue of the local clustering lemma, namely, Theorems \ref{th1.1}, \ref{th1.2}. The proof of Theorems \ref{th1.1} and \ref{th1.2} is based on the precise point-wise estimates of solutions of the nonlinear potential type, see Section \ref{Sec3}  below.

The rest of the paper contains the proof of the above theorems. In Section \ref{Sec2}  we collect some auxiliary
propositions and required integral estimates of solutions. Sections \ref{Sec3} and \ref{Sec4} contain the upper
and lower bounds for auxiliary solutions and proof of Theorems \ref{th1.1}, \ref{th1.2}.
Finally, in Section \ref{Sec5} we give a proof of   Harnack type  inequalities,
Theorems \ref{th1.3}, \ref{th1.4}.

%Section \ref{Sec5} -- , proof of Theorem .

%%%%%%%%%%%%%%%%%%%%%%%%%%%%%%%%%%%%%%%%%%%%%%%%%%%%%%%%%%%%%%%%%%%%%%%%%%%%%%%%%%%%%%%%%%%%%%%%%%%%%%%%%%%%%%%
%%%%%%%%%%%%%%%%%%%%%%%%%%%%%%%%%%%%%%%%%%%%%%%%%%%%%%%%%%%%%%%%%%%%%%%%%%%%%%%%%%%%%%%%%%%%%%%%%%%%%%%%%%%%%%%
%%%%%%%%%%%%%%%%%%%%%%%%%%%%%%%%%%%%%%%%%%%%%%%%%%%%%%%%%%%%%%%%%%%%%%%%%%%%%%%%%%%%%%%%%%%%%%%%%%%%%%%%%%%%%%%%%%

\section{Auxiliary material and integral estimates}\label{Sec2}

\subsection{ Auxiliary Lemmas}\label{subsect2.1}

The following  lemma will be used in the sequel, it is  the well-known De\,Giorgi-Poincar\'{e} Lemma (see \cite[Chap.\,2]{LadUr}).

\begin{lemma}\label{lem2.1}
{\it Let $u \in W^{1,1}(B_{r}(y))$ for some $r > 0$, and $y \in \mathbb{R}^{n}$ . Let $k, l$ be real
numbers such that $k < l$. Then there exists a constant $\gamma$ depending only on $n$ such that
\begin{equation*}
(l-k) |A_{k,r}|\cdot|B_{r}(y)\setminus A_{l,r}| \leqslant \gamma r^{n+1} \int\limits_{A_{l,r}\setminus A_{k,r}} |\nabla u|\,\, dx,
\end{equation*}
where $A_{k,r} = B_{r}(y)\cap \{u < k\}$.
}
\end{lemma}
The following lemma can be also found in \cite[Chap.\,2]{LadUr}.

\begin{lemma}\label{lem2.2}
{\it Let $\{y_{j}\}_{j\in \mathbb{N}}$  be a sequence of non-negative numbers such that for
any $j = 0, 1, 2,...$ the inequality
$$
y_{j+1} \leqslant \gamma c^{j} y_{j}^{1+\varepsilon}
$$
holds with some $\varepsilon,\gamma >0 , c\geqslant 1$. Then the following estimate is true
$$
y_{j} \leqslant \gamma^{\frac{(1+\varepsilon)^{j}-1}{\varepsilon}} c^{\frac{(1+\varepsilon)^{j}-1}{\varepsilon^{2}}-\frac{j}{\varepsilon}}
y_{0}^{(1+\varepsilon)^{j}}.
$$
Particularly, if $y_{0}\leqslant \gamma^{-\frac{1}{\varepsilon}} c^{-\frac{1}{\varepsilon^{2}}},$
then $\lim\limits_{j \rightarrow +\infty} y_{j} =0$.
}
\end{lemma}
The folowing lemma can be found in \cite{Lieberman91}.
\begin{lemma}\label{lem2.3}
Let $Q$ be an increasing, convex, continuous function such that
$Q(0)=0$ and $\frac{1}{s}Q(s)\leqslant \frac{1}{t} Q(t)$  for $0<s\leqslant t.$ If $u\in W^{1,1}(\Omega)$ with $u=0$ on $\partial\Omega$ and $\int\limits_{\Omega} Q(|\nabla u|) dx$ finite, and $R=\rm{diam}\,\, \Omega$  then
$$
\int\limits_{\Omega} Q\left(\frac{u}{R}\right) dx\leqslant  \int\limits_{\Omega} Q\left(|\nabla u|\right) dx.
$$
\end{lemma}

\subsection{ Local energy estimates}\label{subsect2.2}

The following inequality is a simple analogue of Young's inequality

\begin{equation}\label{eq2.1}
g(x, a) b \leqslant \varepsilon a g(x, a) + \max(\varepsilon^{1-p},\varepsilon^{1-q}) b g(x, b),\quad \varepsilon, a, b >0,\quad x\in \Omega,
\end{equation}
indeed, if $b\leqslant \varepsilon a$, then $g (x, a) b \leqslant \varepsilon a g (x, a)$, and if $b \geqslant \varepsilon a$, then by condition $(g)$\quad $g (x, a) b \leqslant\\ \leqslant \max(\varepsilon^{1-p},\varepsilon^{1-q}) b g (x, b)$.

The following lemma, namely inequalities \eqref{eq2.2} and \eqref{eq2.3}, can be used as a definition of the corresponding $DG_{g,\Lambda}(\Omega)$ De Giorgi classes, we refer the reader to \cite{SkrVoitUMB19, SkrVoitNA20} for the case $\Lambda(x_{0}, r) \asymp$ const, $r>0$.

\begin{lemma}\label{lem2.4}
{\it Let $u$ be bounded local weak solution to \eqref{eq1.1}, then for any $B_{8r}(x_{0}) \subset \Omega$, any $k < l$, any $\sigma \in (0,1)$ and any $\zeta(x) \in C_{0}^{\infty}(B_{r}(x_{0})), \zeta(x)=1$ in $B_{(1-\sigma)r}(x_{0})$, $0\leqslant \zeta(x)\leqslant 1$, $|\nabla \zeta(x) | \leqslant \frac{1}{\sigma\,r}$ next inequalities hold
\begin{multline}\label{eq2.2}
\int\limits_{A^{+}_{k,r}\setminus A^{+}_{l,r}}
|\nabla u|\,\zeta^{q}(x)\, dx\leqslant
\gamma \,\sigma^{-q}\,\left(\frac{g^{+}_{B_{r}(x_{0})}\left(\frac{M_{+}(k,r)}{r}\right)}{g^{-}_{B_{r}(x_{0})}
\left(\frac{M_{+}(k,r)}{r}\right)}\right)^{\frac{1}{p}}\,M_{+}(k,r)r^{n-1}\times \\ \times\Lambda(x_{0},r)\left(\frac{|A^{+}_{k,r}\setminus A^{+}_{l,r}|}{|B_{r}(x_{0})|}\right)^{\left(1-\frac{1}
{t(p-1)}\right)\left(1-\frac{1}{p}\right)}\left(\frac{|A^{+}_{k,r}|}{|B_{r}(x_{0})|}\right)^{\left(1-\frac{1}{s}\right)\frac{1}{q}},
\end{multline}
\begin{multline}\label{eq2.3}
\int\limits_{A^{-}_{l,r}\setminus A^{-}_{k,r}}
|\nabla u|\,\zeta^{q}(x)\, dx\leqslant
\gamma \,\sigma^{-q}\,\left(\frac{g^{+}_{B_{r}(x_{0})}\left(\frac{M_{-}(l,r)}{r}\right)}{g^{-}_{B_{r}(x_{0})}
\left(\frac{M_{-}(l,r)}{r}\right)}\right)^{\frac{1}{p}}\,M_{-}(l,r)\,r^{n-1}\times \\ \times\Lambda(x_{0},r)\bigg(\frac{|A^{-}_{l,r}\setminus A^{-}_{k,r}|}{|B_{r}(x_{0})|}\bigg)^{\left(1-\frac{1}
{t(p-1)}\right)\left(1-\frac{1}{p}\right)}\bigg(\frac{|A^{-}_{l,r}|}{|B_{r}(x_{0})|}\bigg)^{\left(1-\frac{1}{s}\right)\frac{1}{q}},
\end{multline}
\begin{multline}\label{eq2.4}
\int\limits_{A^{\pm}_{k,r}}
|\nabla u|\,\zeta^{q}(x)\, dx\leqslant\gamma \,\sigma^{-q}\left(\frac{g^{+}_{B_{r}(x_{0})}\left(\frac{M_{\pm}(k,r)}{r}\right)}{g^{-}_{B_{r}(x_{0})}
\left(\frac{M_{\pm}(k,r)}{r}\right)}\right)^{\frac{1}{p}}\times\\ \times M_{\pm}(k,r)\,r^{n-1}\,\Lambda(x_{0},r)\,\left(\frac{|A^{\pm}_{k,r}|}
{|B_{r}(x_{0})|}\right)^{\left(1-\frac{1}{t(p-1)}\right)\left(1-\frac{1}{p}\right)+\left(1-\frac{1}{s}\right)\frac{1}{q}},
\end{multline}
here $(u-k)_{\pm}:=\max \{\pm(u-k),0 \}$, $A^{\pm}_{k,r}:=B_{r}(x_{0})\cap \{(u-k)_{\pm} >0 \}$, $M_{\pm}(k,r):=\esssup\limits_{B_{r}(x_{0})} (u-k)_{\pm},$ $g^{+}_{B_{r}(x_{0})}(v):=\max\limits_{x\in B_{r}(x_{0})}g(x, v),\quad g^{-}_{B_{r}(x_{0})}(v):=\min\limits_{x\in B_{r}(x_{0})}g(x, v),\quad v>0.$
}
\end{lemma}

\begin{proof}
Test identity \eqref{eq1.4} by $\varphi = (u-k)_{+} \zeta^{q}(x),$ by conditions \eqref{eq1.2} we obtain
\begin{multline*}
\int\limits_{A^{+}_{k,r}} a(x) G(x, | \nabla u|) \zeta^{q}(x) dx \leqslant \gamma \int\limits_{A^{+}_{k,r}} b(x) G\bigg(x, \frac{(u-k)_{+}}{\sigma\,r}\bigg) dx\leqslant\\ \leqslant \gamma \sigma^{-q}\frac{M_{+}(k,r)}{r}g^{+}_{B_{r}(x_{0})}\bigg(\frac{M_{+}(k,r)}{r}\bigg)\int\limits_{A^{+}_{k,r}} b(x) dx,\quad G(x, v):=\int\limits_{0}^{v} g(x, s)\,ds, \quad v>0,
\end{multline*}
and use inequality \eqref{eq2.1} with $ a=\dfrac{M_{+}(k,r)}{r}$, $b= | \nabla u|$ and $\varepsilon $ replaced by $\varepsilon^{-1} a_{1}(x)$, $\varepsilon>0$, integrating over the set $A^{+}_{k,r}\setminus A^{+}_{l,r}$, using condition ($g$)  we arrive at

\begin{multline*}
\int\limits_{A^{+}_{k,r}\setminus A^{+}_{l,r}}
|\nabla u|\,\zeta^{q}(x)\, dx=\int\limits_{A^{+}_{k,r}\setminus A^{+}_{l,r}}\frac{g^{-}_{B_{r}(x_{0})}\big(\frac{M_{+}(k,r)}{r}\big)}{g^{-}_{B_{r}(x_{0})}\big(\frac{M_{+}(k,r)}{r}\big)} |\nabla u|\,\zeta^{q}(x)\, dx\leqslant\qquad\qquad\qquad\qquad\qquad\qquad\\
 \leqslant \varepsilon^{-1}\frac{M_{+}(k,r)}{r}\int\limits_{A^{+}_{k,r}\setminus A^{+}_{l,r}}
[a_{1}(x)]^{\frac{1}{p-1}}dx+ \frac{\max(\varepsilon^{p-1},\varepsilon^{q-1})}{g^{-}_{B_{r}(x_{0})}\big(\frac{M_{+}(k,r)}{r}\big)}\int\limits_{A^{+}_{k,r}} a(x) G(x, | \nabla u|) \zeta^{q}(x) dx\leqslant \\
\leqslant \varepsilon^{-1}\frac{M_{+}(k,r)}{r}\int\limits_{A^{+}_{k,r}\setminus A^{+}_{l,r}}[a_{1}(x)]^{\frac{1}{p-1}}dx+ \qquad\qquad\qquad\qquad\qquad\qquad\qquad\qquad\qquad\\
 + \gamma\,\sigma^{-q} \max(\varepsilon^{p-1},\varepsilon^{q-1})\frac{M_{+}(k,r)}{r}\frac{g^{+}_{B_{r}(x_{0})}
\big(\frac{M_{+}(k,r)}{r}\big)}{g^{-}_{B_{r}(x_{0})}\big(\frac{M_{+}(k,r)}{r}\big)}\int\limits_{A^{+}_{k,r}}b(x)dx \leqslant \qquad\qquad\\
 \leqslant \gamma\,\sigma^{-q}\, M_{+}(k,r) r^{n-1} \left\{\varepsilon^{-1}\left(r^{-n}\int\limits_{B_{r}(x_{0})}a^{t}_{1}(x)dx\right)^{\frac{1}{t(p-1)}}\left(\frac{|A^{+}_{k,r}\setminus A^{+}_{l,r}|}{|B_{r}(x_{0})|}\right)^{1-\frac{1}{t(p-1)}}+\right.
 \end{multline*}
\begin{multline*}
\left.\qquad\qquad\qquad\qquad+\frac{g^{+}_{B_{r}(x_{0})}\big(\frac{M_{+}(k,r)}{r}\big)}{g^{-}_{B_{r}(x_{0})}\big(\frac{M_{+}(k,r)}{r}\big)} \max(\varepsilon^{p-1},\varepsilon^{q-1})\left(r^{-n}\int\limits_{B_{r}(x_{0})}b^{s}(x)dx\right)^{\frac{1}{s}}
\bigg(\frac{|A^{+}_{k,r}|}{|B_{r}(x_{0})|}\bigg)^{1-\frac{1}{s}}\right\}.
\end{multline*}
Choose $\varepsilon$ by the condition
$$\varepsilon^{p}=\frac{g^{+}_{B_{r}(x_{0})}\big(\frac{M_{+}(k,r)}{r}\big)}{g^{-}_{B_{r}(x_{0})}\big(\frac{M_{+}(k,r)}{r}\big)}\frac{\bigg(r^{-n}\int\limits_{B_{r}(x_{0})}a^{t}_{1}(x)dx\bigg)^{\frac{1}{t(p-1)}}}{\bigg(r^{-n}\int\limits_{B_{r}(x_{0})}b^{s}(x)dx\bigg)^{\frac{1}{s}}}\bigg(\frac{|A^{+}_{k,r}\setminus A^{+}_{l,r}|}{|B_{r}(x_{0})|}\bigg)^{1-\frac{1}{t(p-1)}}\bigg(\frac{|A^{+}_{k,r}|}{|B_{r}(x_{0})|}\bigg)^{-1+\frac{1}{s}},$$
if \quad $\dfrac{g^{+}_{B_{r}(x_{0})}\big(\frac{M_{+}(k,r)}{r}\big)}{g^{-}_{B_{r}(x_{0})}\big(\frac{M_{+}(k,r)}{r}\big)}\dfrac{\bigg(r^{-n}\int\limits_{B_{r}(x_{0})}a^{t}_{1}(x)dx\bigg)^{\frac{1}{t(p-1)}}}{\bigg(r^{-n}\int\limits_{B_{r}(x_{0})}b^{s}(x)dx\bigg)^{\frac{1}{s}}}\bigg(\dfrac{|A^{+}_{k,r}\setminus A^{+}_{l,r}|}{|B_{r}(x_{0})|}\bigg)^{1-\frac{1}{t(p-1)}}\bigg(\dfrac{|A^{+}_{k,r}|}{|B_{r}(x_{0})|}\bigg)^{-1+\frac{1}{s}} \leqslant 1$

 and
$$\varepsilon^{q}=\frac{g^{+}_{B_{r}(x_{0})}\big(\frac{M_{+}(k,r)}{r}\big)}{g^{-}_{B_{r}(x_{0})}\big(\frac{M_{+}(k,r)}{r}\big)}\frac{\bigg(r^{-n}\int\limits_{B_{r}(x_{0})}a^{t}_{1}(x)dx\bigg)^{\frac{1}{t(p-1)}}}{\bigg(r^{-n}\int\limits_{B_{r}(x_{0})}b^{s}(x)dx\bigg)^{\frac{1}{s}}}\bigg(\frac{|A^{+}_{k,r}\setminus A^{+}_{l,r}|}{|B_{r}(x_{0})|}\bigg)^{1-\frac{1}{t(p-1)}}\bigg(\frac{|A^{+}_{k,r}|}{|B_{r}(x_{0})|}\bigg)^{-1+\frac{1}{s}},$$
in the opposite case, from the previous we arrive at
\begin{multline*}
\int\limits_{A^{+}_{k,r}\setminus A^{+}_{l,r}}
|\nabla u|\,\zeta^{q}(x)\, dx \leqslant  \gamma\,\sigma^{-q}\Bigg(\frac{g^{+}_{B_{r}(x_{0})}
\big(\frac{M_{+}(k,r)}{r}\big)}{g^{-}_{B_{r}(x_{0})}\big(\frac{M_{+}(k,r)}{r}\big)}\Bigg)^{\frac{1}{p}}M_{+}(k,r) r^{n-1}\Lambda(x_{0},r) \times \\ \times
\bigg\{\bigg(\frac{|A^{+}_{k,r}\setminus A^{+}_{l,r}|}{|B_{r}(x_{0})|}\bigg)^{\left(1-\frac{1}{t(p-1)}\right)\left(1-\frac{1}{p}\right)}\bigg(\frac{|A^{+}_{k,r}|}{|B_{r}(x_{0})|}\bigg)^{\left(1-\frac{1}{s}\right)\frac{1}{p}} +\qquad\qquad\qquad\qquad\qquad\\
+
\bigg(\frac{|A^{+}_{k,r}\setminus A^{+}_{l,r}|}{|B_{r}(x_{0})|}\bigg)^{\left(1-\frac{1}{t(p-1)}\right)\left(1-\frac{1}{q}\right)}
\bigg(\frac{|A^{+}_{k,r}|}{|B_{r}(x_{0})|}\bigg)^{\left(1-\frac{1}{s}\right)\frac{1}{q}}\bigg\}\leqslant \\ \leqslant
\gamma\,\sigma^{-q}\Bigg(\frac{g^{+}_{B_{r}(x_{0})}\big(\frac{M_{+}(k,r)}{r}\big)}{g^{-}_{B_{r}(x_{0})}
\big(\frac{M_{+}(k,r)}{r}\big)}\Bigg)^{\frac{1}{p}}M_{+}(k,r) r^{n-1} \times\\ \times \Lambda(x_{0},r)\bigg(\frac{|A^{+}_{k,r}\setminus A^{+}_{l,r}|}{|B_{r}(x_{0})|}\bigg)^{\left(1-\frac{1}{t(p-1)}\right)\left(1-\frac{1}{p}\right)}
\bigg(\frac{|A^{+}_{k,r}|}{|B_{r}(x_{0})|}\bigg)^{\left(1-\frac{1}{s}\right)\frac{1}{q}},
\end{multline*}
which proves \eqref{eq2.2}. The proof of \eqref{eq2.3} and \eqref{eq2.4} is completely similar.
\end{proof}

\subsection{De\,Giorgi Type Lemmas}\label{subsect2.3}

The following lemmas are De Giorgi type lemmas and their formulation  under condition $(g_{\mu})$ or condition $(g_{\lambda})$ is different. In the proof we closely follow \cite[Chap.\,2]{LadUr}.

\begin{lemma}\rm{\textbf{(De\,Giorgi Type Lemma Under the Condition ($\mathbf{g_{\lambda}}$))}}\label{lem2.5}
{\it Let $u$ be a local bounded weak solution to Eq.\eqref{eq1.1},
$B_{8r}(x_{0})\subset B_{R}(x_{0})\subset \Omega$, assume that condition $(g_{\lambda})$ be fulfilled and
let
$$
\mu^{+}_{r}\geqslant \esssup\limits_{B_{r}(x_{0})}u, \quad
\mu^{-}_{r}\leqslant \essinf\limits_{B_{r}(x_{0})}u, \quad
\omega_{r}:=\mu^{+}_{r}-\mu^{-}_{r}
$$
and set $v_{+}:=\mu^{+}_{r}-u$, $v_{-}:=u-\mu^{-}_{r}$.
Fix $\xi \in(0, \lambda(r))$ and $\delta \in(0,1)$, then there exists $\nu_{1}\in(0,1)$ depending only on the data and $\delta$ such that if
\begin{equation}\label{eq2.5}
\left| \left\{ x\in B_{r}(x_{0}):v_{\pm}(x)\leqslant \xi\, \omega_{r}\right\} \right|
\leqslant\nu_{1}[\Lambda(x_{0}, r)]^{-\frac{1}{\kappa}} |B_{r}(x_{0})|,\quad \kappa= \frac{1}{n}-\frac{1}{tp}-\frac{1}{sq} +\frac{1}{p}-\frac{1}{q},
\end{equation}
then either
\begin{equation}\label{eq2.6}
\delta(1-\delta)\xi \omega_{r} \leqslant r,
\end{equation}
or
\begin{equation}\label{eq2.7}
v_{\pm}(x)\geqslant \delta\,\,\xi \,\omega_{r}
\quad \text{for a.a.} \ x\in B_{r/2}(x_{0}).
\end{equation}}
\end{lemma}
\begin{proof}
We provide the proof of \eqref{eq2.7} for $v_{+}$, while the proof for $v_{-}$
is completely similar.
For $j=0,1,2,\ldots$ we set $r_{j}:=\dfrac{r}{2}(1+2^{-j}), \,\,\bar{r}_{j}=\dfrac{r_{j}+r_{j+1}}{2},\,\,
k_{j}:=\mu^{+}_{r}-\delta(2-\delta) \xi\,\omega_{r}-
\xi\,\omega_{r}\dfrac{(2-\delta)^{2}}{2^{j+1}}$, and let $\zeta_{j} \in C^{\infty}_{0}(B_{\bar{r}_{j}}(x_{0})), \,\,0 \leqslant \zeta_{j} \leqslant 1,\,\, \zeta_{j} =1$ in $B_{r_{j+1}}(x_{0}),$ and $| \nabla \zeta_{j} | \leqslant \gamma \dfrac{2^{j}}{r}.$

Since $M_{+}(k_{j},r)\leqslant \xi\,\omega_{r} \leqslant 2M \lambda(r)$
for $j=0,1,2,\ldots$, so if \eqref{eq2.6} is violated then condition $(g_{\lambda})$ is applicable and
we obtain that
$$
g^{+}_{B_{r}(x_{0})}\left( \frac{M_{+}(k_{j},r_{j})}{r_{j}}  \right)\leqslant \gamma
 g^{-}_{B_{r}(x_{0})}\left( \frac{M_{+}(k_{j},r_{j})}{r}  \right).
$$
Therefore inequality \eqref{eq2.4}  can be rewritten as
\begin{equation*}
\int\limits_{A^{+}_{k_{j},r_{j}}}
|\nabla u| \zeta^{q}_{j}\,dx\leqslant \gamma\,2^{j\gamma}
\xi\,\omega_{r} r^{n-1}\Lambda(x_{0},r)\,\left(\frac{|A^{+}_{k_{j},r_{j}}|}{|B_{r}(x_{0})|}\right)^{\left(1-\frac{1}{t(p-1)}\right)\left(1-\frac{1}{p}\right)+
\left(1-\frac{1}{s}\right)\frac{1}{q}} .
\end{equation*}
From this, using the Sobolev embedding theorem we obtain that
\begin{multline*}
(1-\delta)^{2}\xi\,\frac{\omega_{r}}{2^{j+1}} \left|A^{+}_{k_{j+1},r_{j+1}}\right| \leqslant \int\limits_{A^{+}_{k_{j},\bar{r}_{j}}} (u-k_{j})_{+} \zeta^{q}_{j} dx \leqslant \\ \leqslant \gamma\int\limits_{A^{+}_{k_{j},r_{j}}} | \nabla\bigg((u-k_{j})_{+}\zeta^{q}_{j}(x)\bigg)dx |A^{+}_{k_{j},r_{j}}|^{\frac{1}{n}} \leqslant\qquad\qquad \\
\leqslant \gamma\,2^{j\gamma} \xi\,\omega_{r} r^{n}\Lambda(x_{0},r)\,\bigg(\frac{|A^{+}_{k_{j},r_{j}}|}{|B_{r}(x_{0})|}\bigg)^{\left(1-\frac{1}{t(p-1)}\right)\left(1-\frac{1}{p}\right)+\left(1-\frac{1}{s}\right)\frac{1}{q} +\frac{1}{n}},
\end{multline*}
 from which we arrive at
\begin{equation*}
y_{j+1}=\frac{|A^{+}_{k_{j+1},r_{j+1}}|}{|B_{r}(x_{0})|} \leqslant \gamma\,(1-\delta)^{-2}\, 2^{j\gamma}\Lambda(x_{0},r)
y_{j}^{1+\kappa},\quad j=0,1,2,... .
\end{equation*}
By our choices $\kappa >0$, so choosing $\nu_{1}$ by the condition
$$\nu_{1}(r)=\gamma^{-1}\,\,(1-\delta)^{\frac{2}{\kappa}}$$
and using Lemma \ref{lem2.2}, from this  we arrive at the required \eqref{eq2.7},
which completes the proof of the lemma.
\end{proof}

\begin{lemma}\rm{\textbf{(De\,Giorgi Type Lemma Under the Condition ($\mathbf{g_{\mu}}$))}}\label{lem2.6}
{\it Let $u$ be a local bounded weak solution to Eq.\eqref{eq1.1}, $B_{8r}(x_{0})\subset B_{R}(x_{0})\subset \Omega$, assume that condition $(g_{\mu})$ be fulfilled.
Fix $\xi, \delta\in(0,1)$, then there exists $\nu_{1}\in(0,1)$ depending only on the data and $\delta$ such that if
\begin{equation}\label{eq2.8}
\left| \left\{ x\in B_{r}(x_{0}):v_{\pm}(x)\leqslant \xi\,\omega_{r}\right\} \right|
\leqslant\nu_{1} [\mu(r)]^{-\frac{1}{\kappa}} [\Lambda(x_{0}, r)]^{-\frac{1}{\kappa}}|B_{r}(x_{0})|,
\end{equation}
then either \eqref{eq2.6} holds, or
\begin{equation}\label{eq2.9}
v_{\pm}(x)\geqslant \delta \xi\,\omega_{r}
\quad \text{for a.a.} \ x\in B_{r/2}(x_{0}).
\end{equation}
Here $\kappa >0 $ is the number, defined in Lemma \ref{lem2.5}.}
\end{lemma}
\begin{proof}
For $j=0,1,2,\ldots$ we set $r_{j}:=\dfrac{r}{2}(1+2^{-j}),\,\, \bar{r}_{j}=\dfrac{r_{j}+r_{j+1}}{2},\,\,
k_{j}:=\mu^{+}_{r}- \delta(2-\delta)\xi\omega_{r}-\\
-\dfrac{(1-\delta)^{2}}{2^{j}}\xi\omega_{r}$, and let $\zeta_{j} \in C^{\infty}_{0}(B_{\bar{r}_{j}}(x_{0})),\,\, 0 \leqslant \zeta_{j} \leqslant 1,\,\, \zeta_{j} =1$ in $B_{r_{j+1}}(x_{0}),$ and $| \nabla \zeta_{j} | \leqslant \gamma \dfrac{2^{j}}{r}.$ We assume that
$M_{+}(k_{\infty},r/2)\geqslant \delta(1-\delta)\,\omega_{r}$,
because in the opposite case, the required \eqref{eq2.9} is evident.
If \eqref{eq2.6} is violated then $M_{+}(k_{\infty},r/2)\geqslant  r$.
In addition, since $M_{+}(k_{j},r)\leqslant \xi\,\omega_{r}$
for $j=0,1,2,\ldots$, then condition $(g_{\mu})$ is applicable and
we obtain that
$$
g^{+}_{B_{r}(x_{0})}\left(\frac{M_{+}(k_{j},r_{j})}{r_{j}}  \right)\leqslant \gamma
2^{j\gamma} \mu(r) g^{-}_{B_{r}(x_{0})}\left( \frac{M_{+}(k_{j},r_{j})}{r}  \right).
$$
Therefore inequality \eqref{eq2.4} similarly to Lemma \ref{lem2.5} can be rewritten as
$$
\int\limits_{A^{+}_{k_{j},r_{j}}}
|\nabla u|\,\zeta^{q}_{j}\,dx\leqslant \gamma\,2^{j\gamma}\,\mu(r)
\xi\,\omega_{r} r^{n}\Lambda(x_{0},r)\,\left(\frac{|A^{+}_{k_{j},r_{j}}|}{|B_{r}(x_{0})|}\right)^{\left(1-\frac{1}{t(p-1)}\right)\left(1-\frac{1}{p}\right)+
\left(1-\frac{1}{s}\right)\frac{1}{q}} .
$$

From this, using the Sobolev embedding theorem, similarly to Lemma \ref{lem2.4} we obtain
$$
y_{j+1}=\frac{|A^{+}_{k_{j+1},r_{j+1}}|}{|B_{r}(x_{0})|} \leqslant \gamma\,2^{j\gamma}\,(1-\delta)^{-2}\,\mu(r)\,\Lambda(x_{0}, r) y_{j}^{1+\kappa}, \,\,j=0,1,2,... ,
$$
choosing $\nu_{1}$ from the condition
$$\nu_{1}=\gamma^{-1}\,(1-\delta)^{\frac{2}{\kappa}}$$
and using Lemma \ref{lem2.2}, from this  we arrive at the required \eqref{eq2.9},
which completes the proof of the lemma.
\end{proof}

\subsection{ Expansion of the Positivity Lemma}\label{subsect2.4}
The following two lemmas will be used in the sequel. In the proof we closely follow to \cite[Chap.\,2]{LadUr}. The first one is a variant of the expansion of positivity lemma under the $(g_{\lambda})$ condition.
\begin{lemma}\label{lem2.7}
Let $u$ be a local bounded weak solution to Eq. \eqref{eq1.1},
let $B_{8r}(x_{0})\subset B_{R}(x_{0})\subset \Omega$ and $\xi  \in(0,1)$,\,in addition let condition $(g_{\lambda})$ be fulfilled.
Assume that with some $\alpha_{0}\in(0,1)$ and $\bar{\lambda}(r)\in(0,\lambda(\frac{r}{2}))$ there holds
\begin{equation}\label{eq2.10}
\left| \left\{ x\in B_{3r/4}(x_{0}): v_{\pm}(x)
\leqslant \xi\,\bar{\lambda}(r)\,\omega_{r} \right\} \right|
\leqslant(1-\alpha_{0})\, |B_{3r/4}(x_{0})|.
\end{equation}
Then there exists number $C_{\ast}$ depending only on the known data, $\alpha_{0}$ and $\xi$, such that
either
\begin{equation}\label{eq2.11}
\omega_{r} \leqslant \frac{r}{\bar{\lambda}(r)}\,\,\exp\left(C_{\ast} [\Lambda(x_{0},r)]^{\bar{\beta}_{1}}\right),
\end{equation}
or
\begin{equation}\label{eq2.12}
v_{\pm}(x)\geqslant \omega_{r} \bar{\lambda}(r)\, \exp\left(-C_{\ast} [\Lambda(x_{0}, r)]^{\bar{\beta}_{1}}\right)
\quad \text{for a.a.} \ x\in B_{r/2}(x_{0}).
\end{equation}
Here $\bar{\beta}_{1} >0$ is some fixed  number depending only upon the data.
\end{lemma}
\begin{proof}
We provide the proof of \eqref{eq2.12} for $v_{+}$,
while the proof for $v_{-}$ is completely similar.
We set $k_{j}:=\mu_{r}^{+}-\dfrac{\bar{\lambda}(r)\,\omega_{r}}{2^{j}}$,
$j=[\log 1/\xi]+1,2, \ldots,j_{\ast}$, where $j_{\ast}$ is large enough to be chosen later. We will assume that
$M_{+}(k_{j_{\ast}},\frac{r}{2}) \geqslant \frac{\omega_{r}}{2^{j_{*}+1}}$, because in the opposite case, the required
\eqref{eq2.12} is evident. If \eqref{eq2.11} is violated, then $M_{+}(k_{j_{\ast}},\frac{r}{2}) \geqslant r$, and
since $M_{+}(k_{j},r)\leqslant 2^{-j}\bar{\lambda}(r)\,\omega_{r} \leqslant 2M \lambda(r)$, $j=\overline{[\log 1/\xi]+1,j_{\ast}}$,
then by  $(g_{\lambda})$ we obtain that
$$
g^{+}_{B_{r}(x_{0})}\left( \frac{M_{+}(k_{j},r)}{r} \right)\leqslant
g^{-}_{B_{r}(x_{0})}\left( \frac{M_{+}(k_{j},r)}{r} \right).
$$
Therefore inequality \eqref{eq2.2}
can be rewritten as
\begin{equation*}
\int\limits_{A^{+}_{k_{j},r}\setminus A^{+}_{k_{j+1},r}}
|\nabla u|\,\zeta^{\,q}\,dx\leqslant \gamma\,\bar{\lambda}(r)\,\frac{\omega_{r}}{2^{j}}\,r^{n-1}\Lambda(x_{0},r)\Bigg(\frac{|A^{+}_{k_{j},r}\setminus A^{+}_{k_{j+1},r}|}{|B_{r}(x_{0})|}\bigg)^{\left(1-\frac{1}{t(p-1)}\right)(1-\frac{1}{p})},
\end{equation*}
where $\zeta\in C_{0}^{\infty}(B_{r}(x_{0}))$, $0\leqslant\zeta\leqslant1$,
$\zeta=1$ in $B_{3r/4}(x_{0})$, $|\nabla\zeta|\leqslant 4/r$.
From this, using \eqref{eq2.10}, De\,Giorgi-Poincar\'{e} Lemma \ref{lem2.1}, we obtain
\begin{multline*}
\bar{\lambda}(r)\,\frac{\omega_{r}}{2^{j+1}}|A^{+}_{k_{j},\frac{3}{4}r}|\leqslant \frac{\gamma}{\alpha_{0}}\,r\,\int\limits_{A^{+}_{k_{j},r}\setminus A^{+}_{k_{j+1},r}}|\nabla u|\,\zeta^{\,q}\,dx\leqslant \\ \leqslant \frac{\gamma}{\alpha_{0}}\,\bar{\lambda}(r)\,\frac{\omega_{r}}{\,2^{j}}\,r^{n-1}\Lambda(x_{0},r)\left(\frac{|A^{+}_{k_{j},r}\setminus A^{+}_{k_{j+1},r}|}{|B_{r}(x_{0})|}\right)^{\kappa_{1}},\quad
\kappa_{1}=\left(1-\frac{1}{t(p-1)}\right)\left(1-\frac{1}{p}\right),
\end{multline*}
raising the left and right hand-sides to the power $\dfrac{1}{\kappa_{1}}$ and summing up the resulting inequalities in $j$, $j=[\log 1/\xi]+1,2, \ldots,j_{\ast}$ ,we conclude that
$$
(j_{\ast} -[\log1/ \xi]-1)\left|A^{+}_{k_{j_{\ast}},\frac{3}{4}r}\right|^{\frac{1}{\kappa_{1}}}\leqslant \gamma\,\alpha_{0}^{-\frac{1}{\kappa_{1}}} [\Lambda(x_{0},r)]^{\frac{1}{\kappa_{1}}} \left|B_{\frac{3}{4}r}(x_{0})\right|^{\frac{1}{\kappa_{1}}}.
$$
Choosing $j_{\ast}$ by the condition
$$\gamma^{\kappa_{1}}\,\alpha_{0}^{-1}\,(j_{\ast} -[\log1/ \xi]-1)^{-\kappa_{1}} \leqslant \nu_{1}\,[\Lambda(x_{0},r)]^{-1-\frac{1}{\kappa}},$$where $\nu_{1}$ and $\kappa$ are the constants, defined in Lemma \ref{lem2.5}, therefore by \eqref{eq2.7} we obtain inequality \eqref{eq2.12}, which proves Lemma \ref{lem2.7} with $\bar{\beta}_{1}= \frac{1}{\kappa_{1}}\,(1+\frac{1}{\kappa})$.
\end{proof}

The following lemma is a variant of the expansion of positivity lemma under the $(g_{\mu})$ condition.
\begin{lemma}\label{lem2.8}
Let $u$ be a local bounded weak solution to Eq.\eqref{eq1.1},
let $B_{8r}(x_{0})\subset B_{R}(x_{0})\subset \Omega$ and $\xi  \in(0,1)$,\, in addition let condition $(g_{\mu})$ be fulfilled.
Assume that with some $\alpha_{0}\in(0,1)$ there holds
\begin{equation}\label{eq2.13}
\left| \left\{ x\in B_{3r/4}(x_{0}):v_{\pm}(x)
\leqslant \xi\,\omega_{r} \right\} \right|
\leqslant(1-\alpha_{0})\, |B_{3r/4}(x_{0})|.
\end{equation}
Then there exists number $C_{\ast}$ depending only on the known data, $\alpha_{0}$ and $\xi$ such that either
\begin{equation}\label{eq2.14}
\omega_{r} \leqslant \,r\, \exp\big(C_{\ast} [\mu(r)]^{\bar{\beta}_{1}}\,[\Lambda(x_{0}, r)]^{\bar{\beta}_{1}}\big),
\end{equation}
\begin{equation}\label{eq2.15}
v_{\pm}(x)\geqslant \xi \omega_{r} \exp \big(-C_{\ast}[\mu(r)]^{\bar{\beta}_{1}}\,[\Lambda(x_{0}, r)]^{\bar{\beta}_{1}}\big) ,
\quad \text{for a.a.} \ x\in B_{r/2}(x_{0}),
\end{equation}
here $\bar{\beta}_{1} >0$ is the number, defined in Lemma \ref{lem2.7}.
\end{lemma}
\begin{proof}
We provide the proof of \eqref{eq2.15} for $v_{+}$,
while the proof for $v_{-}$ is completely similar.
We set $k_{j}:=\mu_{r}^{+}-\dfrac{\omega_{r}}{2^{j}}$,
$j=[\log 1/\xi]+1,2, \ldots,j_{\ast}$, where $j_{\ast}$ to be chosen. We will assume that
$M_{+}(k_{j_{\ast}},\frac{r}{2}) \geqslant \dfrac{\omega_{r}}{2^{j_{\ast}+1}}$, because in the opposite case, the required \eqref{eq2.15}
is evident. If \eqref{eq2.14} is violated, then $M_{+}(k_{j_{\ast}},\frac{r}{2}) \geqslant r,$ \,and
since
$M_{+}(k_{j},r)\leqslant 2^{-j}\,\omega_{r}$, $j=\overline{[\log 1/\xi]+1,j_{\ast}}$,
then by  $(g_{\mu})$ we obtain that
$$
g^{+}_{B_{r}(x_{0})}\left( \frac{M_{+}(k_{j},r)}{r} \right)\leqslant \gamma \mu(r)
g^{-}_{B_{r}(x_{0})}\left( \frac{M_{+}(k_{j},r)}{r} \right).
$$
Therefore, completely similar to the previous lemma, inequality \eqref{eq2.2}
can be rewritten as
$$
\int\limits_{A^{+}_{k_{j},r}\setminus A^{+}_{k_{j+1},r}}
|\nabla u|\,\zeta^{\,q}\,dx\leqslant \frac{\gamma}{\alpha_{0}}\,\mu(r)\frac{\omega_{r}}{2^{j}}\,r^{n-1}
\Lambda(x_{0}, r)\left( \dfrac{|A^{+}_{k_{j},r}\setminus A^{+}_{k_{j+1},r}|}{|B_{r}(x_{0})|} \right)^{\kappa_{1}},
$$
where $\kappa_{1}=\left(1-\frac{1}{t(p-1)}\right)\left(1-\frac{1}{p}\right)$ and $\zeta\in C_{0}^{\infty}(B_{r}(x_{0}))$, $0\leqslant\zeta\leqslant1$,
$\zeta=1$ in $B_{3r/4}(x_{0})$, $|\nabla\zeta|\leqslant 4/r$.
From this, using \eqref{eq2.13} and De\,Giorgi-Poincar\'{e} Lemma \ref{lem2.1} we obtain
\begin{equation*}
\frac{\omega_{r}}{2^{j}}\left|A^{+}_{k_{j},\frac{3}{4}r}\right|\leqslant \gamma\,r\,\int\limits_{A^{+}_{k_{j},r}\setminus A^{+}_{k_{j+1},r}}
|\nabla u|\,\zeta^{\,q}\,dx\leqslant \frac{\gamma}{\alpha_{0}}\,\mu(r)\,\frac{\omega_{r}}{\,2^{j}}\,r^{n}\,\Lambda(x_{0}, r)
\left( \dfrac{|A^{+}_{k_{j},r}\setminus A^{+}_{k_{j+1},r}|}
{|B_{r}(x_{0})|} \right)^{\kappa_{1}},
\end{equation*}
raising the left and right hand-sides to the power $\dfrac{1}{\kappa_{1}}$ and summing up the resulting inequalities in $j$, $j=[\log 1/\xi]+1,2, \ldots,j_{\ast}$ ,we conclude that
$$
(j_{\ast} -[\log1/ \xi]-1)\left|A^{+}_{k_{j_{\ast}},\frac{3}{4}r}\right|^{\frac{1}{\kappa_{1}}}\leqslant \gamma\,\alpha_{0}^{-\frac{1}{\kappa_{1}}}\,[\mu(r)]^{\frac{1}{\kappa_{1}}} [\Lambda(x_{0}, r)]^{\frac{1}{\kappa_{1}}}\left|B_{\frac{3}{4}r}(x_{0})\right|^{\frac{1}{\kappa_{1}}}.
$$

Choosing $j_{\ast}$ by the condition
$$\gamma^{\kappa_{1}}\,\alpha_{0}^{-1}\,(j_{\ast} -[\log1/ \xi]-1)^{-\kappa_{1}} \leqslant \nu_{1}\,[\mu(r)]^{-1-\frac{1}{\kappa}}\,[\Lambda(x_{0}, r)]^{-1-\frac{1}{\kappa}},$$

by Lemma \ref{lem2.6}  we obtain inequality \eqref{eq2.15}, which proves Lemma \ref{lem2.8}.
\end{proof}

%%%%%%%%%%%%%%%%%%%%%%%%%%%%%%%%%%%%%%%%%%%%%%%%%%%%%%%%%%%%%%%%%%%%%%%%%%%%%%%%%%%%%%%%%%%%%%%%%%%%%%%%%%%%%%%%%
%%%%%%%%%%%%%%%%%%%%%%%%%%%%%%%%%%%%%%%%%%%%%%%%%%%%%%%%%%%%%%%%%%%%%%%%%%%%%%%%%%%%%%%%%%%%%%%%%%%%%%%%%%%%%%%%
%%%%%%%%%%%%%%%%%%%%%%%%%%%%%%%%%%%%%%%%%%%%%%%%%%%%%%%%%%%%%%%%%%%%%%%%%%%%%%%%%%%%%%%%%%%%%%%%%%%%%%%%%%%%%%%%%%

\section{Upper and lower estimates of auxiliary solutions under ($\mathbf{g_{\lambda}}$) condition.
Proof of Theorem \ref{th1.1}}\label{Sec3}
Fix $x_{0}\in \Omega$, let $E\subset B_{r}(x_{0})\subset B_{\rho}(x_{0})\subset B_{R}(x_{0})\subset\Omega$ and consider
the solution $v$ of the following problem
\begin{equation*}
{\rm div} \bigg( H(x, |\nabla v|) \frac{\nabla v}{ |\nabla v|^{2}}\bigg)=0, \quad x\in \mathcal{D}=B_{8\rho}(x_{0})\setminus E, \quad
v- m \psi \in W_{0}(\mathcal{D}),
\end{equation*}
where $m\in [\rho, \lambda(\rho)]$ is some fixed number and $\psi \in W_{0}(\mathcal{D}), \psi=1$ on $E$.

We note that our assumptions on the function $h(x,\cdot)$ imply the monotonicity condition
\begin{equation*}
\bigg( H(x, |\xi|) \frac{\xi}{|\xi|^{2}} -H(x, |\eta|) \frac{\eta}{|\eta|^{2}} ,\,\, \xi- \eta \bigg) >0,\quad \xi, \eta \in \mathbb{R}^{n},\quad \xi \ne \eta,
\end{equation*}
indeed,  by the Cauchy inequality and since $h(x,\cdot)$ increases
\begin{multline*}
\bigg( H(x, |\xi|) \frac{\xi}{|\xi|^{2}} -H(x, |\eta|) \frac{\eta}{|\eta|^{2}} ,\,\, \xi- \eta \bigg)=\\=
\left(\int\limits_{0}^{1}\frac{d}{dt}\left[H(x,|t \xi +(1-t) \eta|)) \frac{t \xi +(1-t) \eta}{|t \xi +(1-t) \eta|^{2}}\right] dt,\,\, \xi-\eta \right)
= \\ = |\xi-\eta|^{2} \int\limits_{0}^{1} \frac{H(x,|t \xi +(1-t) \eta)|)}{|t \xi +(1-t) \eta|^{2}}\,dt+
\int\limits_{0}^{1}\frac{h(x,|t \xi +(1-t) \eta|)}{|t \xi +(1-t) \eta|^{3}}\big|(t \xi+ (1-t) \eta, \,\, \xi-\eta)\big|^{2}\,dt - \\
- 2\int\limits_{0}^{1}\frac{H(x,|t \xi +(1-t) \eta|)}{|t \xi +(1-t) \eta|^{4}}\big|(t \xi+ (1-t) \eta,\,\, \xi-\eta)\big|^{2}\,dt \geqslant
\\ \geqslant \int\limits_{0}^{1} \bigg[\frac{h(x,|t \xi +(1-t) \eta|)}{|t \xi +(1-t) \eta|^{3}}-\frac{H(x,|t \xi +(1-t) \eta|)}{|t \xi +(1-t) \eta|^{4}}\bigg] \big|(t \xi+ (1-t) \eta,\,\, \xi-\eta)\big|^{2}\,dt > 0.
\end{multline*}
So, the existence of the solutions
$v$ follows from the general theory of monotone operators. We will assume that the following
integral identity holds:
\begin{equation}\label{eq3.1}
\int\limits_{\mathcal{D}} H(x, |\nabla v|)\,\frac{\nabla v}{|\nabla v|^{2}} \nabla\varphi \,dx=0
\quad \text{for any } \ \varphi\in W_{0}(\mathcal{D}).
\end{equation}

Testing \eqref{eq3.1} by $\varphi=(v-m)_{+}$ and  by $\varphi=v_{-}$  and using condition
\eqref{eq1.4}, we obtain that $0\leqslant v\leqslant m$.

To formulate our next result, we need the notion of the capacity. For this set
$$
C_{H}(E, B_{8\rho}(x_{0});m):=\,\frac{1}{m}\,\inf\limits_{\varphi\in \mathfrak{M}(E)}
\int\limits_{B_{8\rho}(x_{0})}H(x,m|\nabla \varphi|)\,dx ,
$$
where the infimum is taken over the set $\mathfrak{M}(E)$   of all functions
$\varphi\in W_{0}(B_{8\rho}(x_{0}))$ with $\varphi\geqslant1$ on $E$. If $m=1$,
this definition leads to the standard definition of $C_{H}(E, B_{8\rho}(x_{0}))$
capacity (see, e.g. \cite{HarHastZAn19}).

Further we will assume that
\begin{equation}\label{eq3.2}
g^{-1}_{x_{0}}\bigg(\frac{C_{H}(E, B_{8\rho}(x_{0}); m)}{\rho^{n-1}\,[\Lambda\big(x_{0},\frac{\rho}{4}\big)]^{\bar{c}_{1}}}\bigg) \geqslant \bar{c},
\end{equation}
where $\bar{c}, \bar{c}_{1} > 0 $ to be chosen later depending only on the data.

%%%%%%%%%%%%%%%%%%%%%%%%%%%%%%%%%%%%%%%%%%%%%%%%%%%%%%%%%%%%%%%%%%%%%%%%%%%%%%%%%%%%%%%%%%%%%%%%%%%%%%%%%%%%%%%%%%%%
%%%%%%%%%%%%%%%%%%%%%%%%%%%%%%%%%%%%%%%%%%%%%%%%%%%%%%%%%%%%%%%%%%%%%%%%%%%%%%%%%%%%%%%%%%%%%%%%%%%%%%%%%%%%%%%%%%%%

\subsection{Upper bound for the function $\mathbf{v}$}
We note that in the standard case (i.e. if $p=q$) the upper bound for the function $v$
was proved in \cite{IVSkr1983}
(see also \cite[Chap.\,8, Sec.\,3]{IVSkrMetodsAn1994}, \cite{IVSkrSelWorks}).

\begin{lemma}\label{lem3.1}
There exists $\bar{\beta} >0$ depending only on the data such that
\begin{equation*}
v(x)\leqslant \gamma\,\left[\Lambda\left(x_{0}, \frac{\rho}{4}\right)\right]^{\bar{\beta}}\,\,\rho\,g^{-1}_{x_{0}}\bigg(\frac{C_{H}(E, B_{8\rho}(x_{0}); m)}{\rho^{n-1}}\bigg),\quad x\in B_{\rho}(x_{0})\setminus B_{\frac{\rho}{2}}(x_{0}).
\end{equation*}
\end{lemma}
For $i, j=0,1,2, \ldots$ set $k_{j}:= k(1-2^{-j})$, $k>0$ to be chosen later,
$\rho_{i,j}:=2^{-i-j-3}\rho$,
$$
M_{i}:=\esssup\limits_{F_{i}}v, \quad
F_{i}:=\left\{ x\in \mathcal{D}: \frac{\rho}{4}(1+2^{-i})\leqslant |x-x_{0}|
\leqslant \frac{\rho}{2}(3-2^{-i}) \right\}.
$$
Fix $\overline{x}\in F_{i}$ and suppose that $(v(\overline{x})-k)_{+}\geqslant\rho$,
then $M_{i,j}(k_{j}):=\esssup\limits_{B_{\rho_{i,j}}(\overline{x})}(v-k_{j})\geqslant (v(\overline{x})-k)_{+}\geqslant\\ \geqslant \rho\geqslant \rho_{i,j}$.
And let
$\zeta_{i,j}\in C_{0}^{\infty}(B_{\rho_{i,j}}(\overline{x})), \ \
0\leqslant\zeta_{i,j}\leqslant 1, \ \
\zeta_{i,j}=1 \ \text{in} \ B_{\rho_{i,j+1}}(\overline{x}), \ \
|\nabla \zeta_{i,j}|\leqslant 2^{i+j+4}/\rho.$

By our choice  we have
$v(x)\leqslant m\leqslant 2\,M\,\lambda(\rho)\leqslant\gamma\, 2^{(i+j)\gamma}\,\lambda(\rho_{i,j})$,
$x\in B_{\rho_{i,j}}(\overline{x})$.
Therefore condition ($ g_{\lambda}$) is applicable in $B_{\rho_{i,j}}(\overline{x})$ and we have
by ($ g_{\lambda}$) with $K=\gamma\, 2^{(i+j)\gamma}$ that
$$
g^{+}_{B_{\rho_{i,j}}(\overline{x})}\left( \frac{M_{i,j}(k_{j+1})}{\rho_{i,j}}\right)
\leqslant 2^{(i+j)\gamma}
g^{-}_{B_{\rho_{i,j}}(\overline{x})}\left(\frac{M_{i,j}(k_{j+1})}{\rho_{i,j}}\right).
$$
So, by \eqref{eq2.4} we obtain
\begin{multline*}
\int\limits_{B_{\rho_{i,j}}(\overline{x})}
|\nabla(v-k_{j+1})_{+}|\,\zeta_{i,j}^{\,q}\,dx
\leqslant
\gamma\,2^{\gamma(i+j)}\,M_{i+1}\,\Lambda\left(x_{0}, \frac{\rho}{4}\right)\times\\
\times\rho^{n-1}
\,\bigg(\frac{|A^{+}_{\rho_{i,j}, k_{j+1}}|}{|B_{\rho_{i,j}}(\overline{x})|}\bigg)^{\left(1-\frac{1}{t(p-1)}\right)\left(1-\frac{1}{p}\right)+\left(1-\frac{1}{s}\right)\frac{1}{q}} \leqslant \qquad\qquad\qquad\qquad\\ \leqslant \gamma\,2^{\gamma(i+j)}
M_{i+1}\,\Lambda\left(x_{0}, \frac{\rho}{4}\right)\,\rho^{-n\kappa}\,k^{- 1- \kappa +\frac{1}{n}}\,\left(\,\,\int\limits_{B_{\rho_{i,j}}(\overline{x})} (v-k_{j})_{+}\,dx\right)^{1+\kappa -\frac{1}{n}},
\end{multline*}
here $A^{+}_{\rho_{i,j}, k_{j+1}}:=
B_{\rho_{i,j}}(\overline{x})\cap \{v>k_{j+1}\}$ and $\kappa= \dfrac{1}{n}-\dfrac{1}{tp}-\dfrac{1}{sq}-\dfrac{1}{p}+\dfrac{1}{q} > 0$.

From this, similarly to that of Section \ref{Sec2} and
choosing $k$ from the condition
$$
k^{1+\frac{1}{\kappa}}= \gamma\,2^{i\gamma}M_{i+1}^{\frac{1}{\kappa}} \left[\Lambda\left(x_{0}, \frac{\rho}{4}\right)\right]^{\frac{1}{\kappa}}\,\rho^{-n}\int\limits_{B_{\rho/2^{i+3}}(\overline{x})} v\,dx ,
$$
since $\overline{x}\in F_{i}$ is an arbitrary point, using the Young inequality and keeping in mind our assumption that $v(\overline{x})\geqslant k+\rho$ ,
from the previous we obtain for any $\varepsilon\in(0,1)$
\begin{equation}\label{eq3.3}
M_{i}\leqslant \varepsilon M_{i+1}+ \frac{\gamma\,2^{i\gamma}}{\varepsilon^{\frac{1}{\kappa}}\rho^{\,n}}\,\left[\Lambda\left(x_{0}, \frac{\rho}{4}\right)\right]^{\frac{1}{\kappa}}\,\int\limits_{F_{i+1}} v\,dx+\gamma\rho,
\quad i=0,1,2, \ldots.
\end{equation}

Let us estimate the second term on the right-hand side of \eqref{eq3.3}. For this we assume that \\ $M_{0}\geqslant \rho$, because otherwise,
by \eqref{eq3.2} the upper estimate  is evident. Since $\rho\leqslant M_{0}\leqslant M_{i}\leqslant$ $ \leqslant  2\,M\,\lambda(\rho), i=0,1,2,...$, condition {\rm ($g_{\lambda}$)} is applicable. Set $v_{M_{i+1}}:=\min\{v, M_{i+1}\}$,
by \eqref{eq2.1} and $(g_{\lambda})$ with $K=2M$ we have with arbitrary $\varepsilon_{1}\in(0,1)$
\begin{multline*}
\int\limits_{F_{i+1}} v\,dx=\int\limits_{F_{i+1}} v_{M_{i+1}}\,dx
\leqslant \gamma\rho\int\limits_{\mathcal{D}} |\nabla v_{M_{i+1}}|\,dx
\\
= \gamma\rho\int\limits_{\mathcal{D}}
|\nabla v_{M_{i+1}}|\,\frac{g^{-}_{B_{8\rho}(x_{0})}(M_{i}/\rho)}{g^{-}_{B_{8\rho}(x_{0})}(M_{i}/\rho)}\,dx
\leqslant
\varepsilon_{1} M_{i}\,\rho^{n}+
 \frac{\gamma\varepsilon_{1}^{1-q}\Lambda\big(x_{0},\frac{\rho}{4}\big)\rho}{g^{-}_{B_{8\rho}(x_{0})}(M_{i}/\rho)}
\int\limits_{\mathcal{D}} a(x)G(x,|\nabla v_{M_{i+1}}|)\,dx \leqslant \\
\leqslant \varepsilon_{1} M_{i}\,\rho^{n}+\frac{\gamma\varepsilon_{1}^{1-q}\Lambda\big(x_{0},\frac{\rho}{4}\big)\rho}{g(x_{0},M_{i}/\rho)}\,\int\limits_{\mathcal{D}} H (x,|\nabla v_{M_{i+1}}|)\,dx.
\end{multline*}
Collecting the last two inequalities and choosing $\varepsilon_{1}$ from the condition
$$\gamma\,2^{i\gamma}\,\left[\Lambda\left(x_{0}, \frac{\rho}{4}\right)\right]^{\frac{1}{\kappa}}\varepsilon^{-\frac{1}{\kappa}}\varepsilon_{1}=\varepsilon,$$we obtain
$$
M_{i}\leqslant 2\varepsilon M_{i+1}+\gamma\,2^{i\gamma}\varepsilon^{-\gamma}
\frac{\big[\Lambda\big(x_{0}, \frac{\rho}{4}\big)\big]^{1+\frac{q-1}{\kappa}}\,\rho^{1-n}}{g(x_{0}, M_{i}/\rho)}
\int\limits_{\mathcal{D}} H (x,|\nabla v_{M_{i+1}}|)\,dx+\gamma\rho,
$$
which by \eqref{eq2.1} implies
\begin{equation}\label{eq3.4}
\begin{aligned}
M_{i}\,&g(x_{0},M_{i}/\rho)\leqslant
3\varepsilon M_{i+1}\,g(x_{0},M_{i+1}/\rho)
\\
&+
\gamma\,2^{i\gamma}\varepsilon^{-\gamma}\,\left[\Lambda\left(x_{0},\frac{\rho}{4}\right)\right]^{1+\frac{q-1}{\kappa}}\,\rho^{1-n}\int
\limits_{\mathcal{D}}H (x,|\nabla v_{M_{i+1}}|)\,dx
+\gamma\rho, \quad i=0,1,2,\ldots .
\end{aligned}
\end{equation}
Let $\psi \in W_{0}(\mathcal{D})$ be such that
$\frac{1}{m} \int\limits_{\mathcal{D}} H(x, m|\,\nabla \psi|)\,dx \leqslant C_{H}( E, B_{8\rho}(x_{0}), m) + \gamma \rho^{n}$.
Testing identity \eqref{eq3.1} by $\varphi=v-m\psi$,
by the Young inequality \eqref{eq2.1} and using condition $(g_{\lambda})$ we obtain
\begin{equation}\label{eq3.5}
\int\limits_{\mathcal{D}} H (x,|\nabla v|)\,dx
\leqslant \gamma m \int\limits_{B_{2r}(x_{0})}
H (x,m|\nabla \psi|)\,|\nabla \psi|\,dx\leqslant \gamma\,m
\big( C_{H}(E, B_{8\rho}(x_{0}), m) +\rho ^{n}\big).
\end{equation}
Testing \eqref{eq3.1} by $\varphi=v_{M_{i+1}}-M_{i+1}v/m$ and using the Young inequality \eqref{eq2.1} and \eqref{eq3.5}, we have
$$
\int\limits_{\mathcal{D}} H (x,|\nabla v_{M_{i+1}}|)\,dx
\leqslant \frac{\gamma M_{i+1}}{m}
\int\limits_{\mathcal{D}} H (x,|\nabla v|)\,dx
\leqslant
\gamma M_{i+1}\,\big(C_{H}(E, B_{8\rho}(x_{0}), m) +\rho ^{n}\big).
$$
This inequality and \eqref{eq3.4} imply that
$$
\begin{aligned}
M_{i}\,g(x_{0},M_{i}/\rho)&\leqslant
3\varepsilon M_{i+1}\,g(x_{0},M_{i+1}/\rho)
\\
&+
\gamma\,2^{i\gamma}\varepsilon^{-\gamma} M_{i+1}\,\left[\Lambda\left(x_{0},\frac{\rho}{4}\right)\right]^{1+\frac{q-1}{\kappa}}\bigg( \rho^{1-n}\,C_{H}(E, B_{8\rho}(x_{0}), m) +\rho \bigg),
\quad i=0,1,2,\ldots ,
\end{aligned}
$$
which yields for any $\varepsilon_{2}\in (0,1)$
\begin{multline*}
g(x_{0},M_{i}/\rho) \leqslant
\frac{1}{\varepsilon_{2}}\, \frac{M_{i}}{M_{i+1}}\, g(x_{0},M_{i}/\rho)
+\varepsilon_{2}^{p-1} g(x_{0},M_{i+1}/\rho) \leqslant
\left(\frac{3\varepsilon}{\varepsilon_{2}}+\varepsilon_{2}^{\,p-1}\right)
g(x_{0},M_{i+1}/\rho)+
\\
+\frac{\gamma\,2^{i\gamma}}{(\varepsilon\varepsilon_{2})^{\gamma}}\,\left[\Lambda\left(x_{0},\frac{\rho}{4}\right)\right]^{1+\frac{q-1}{\kappa}}
\left( \rho^{1-n}\,C_{H}(E, B_{8\rho}(x_{0}), m) +\rho \right) \leqslant
\\
\leqslant
\left(\frac{3\varepsilon}{\varepsilon_{2}}+\varepsilon_{2}^{\,p-1}\right)
g(x_{0},M_{i+1}/\rho)+
\frac{\gamma\,2^{i\gamma}}{(\varepsilon\varepsilon_{2})^{\gamma}}\,\left[\Lambda\left(x_{0},\frac{\rho}{4}\right)\right]^{1+\frac{q-1}{\kappa}}\,\rho^{1-n}\,C_{H}(E, B_{8\rho}(x_{0}), m)
,
\quad i=0,1,2,\ldots,
\end{multline*}
provided that $\bar{c}$ in  \eqref{eq3.2} is chosen to satisfy $\bar{c} \geqslant 1$.
Iterating the last inequality and choosing $\varepsilon_{2}$ and $\varepsilon$ small enough,
we arrive at
$$
g(x_{0},M_{0}/\rho)\leqslant \gamma\,\left[\Lambda\left(x_{0},\frac{\rho}{4}\right)\right]^{1+\frac{q-1}{\kappa}}\,\rho^{1-n}\,
C_{H}(E, B_{8\rho}(x_{0}), m),
$$
which proves the upper bound of the function $v$ with
$$\bar{\beta}=\frac{1}{p-1}\,\left(1 +\frac{q-1}{\kappa}\right).$$

%%%%%%%%%%%%%%%%%%%%%%%%%%%%%%%%%%%%%%%%%%%%%%%%%%%%%%%%%%%%%%%%%%%%%%%%%%%%%%%%%%%%%%%%%%%%%%%%%%%%%%%%%%%%%%%%%
%%%%%%%%%%%%%%%%%%%%%%%%%%%%%%%%%%%%%%%%%%%%%%%%%%%%%%%%%%%%%%%%%%%%%%%%%%%%%%%%%%%%%%%%%%%%%%%%%%%%%%%%%%%%%%%%

\subsection{Lower bound for the function $\mathbf{v}$}

The main step in the proof of the lower bound is the following lemma.
\begin{lemma}\label{lem3.2}
There exist  numbers $\varepsilon$, $\vartheta\in(0,1)$, $\bar{\beta}_{2}, \bar{\beta}_{3} >0$ depending only on the data
such that
\begin{multline}\label{eq3.6}
\left| \left\{K_{\rho/4,\, 2\rho}:
v(x)\leqslant \varepsilon\,m\, \left[\Lambda\left(x_{0}, \frac{\rho}{4}\right)\right]^{-\bar{\beta}_{2}}\rho\, g^{-1}_{x_{0}}
\left( \frac{C_{H}(E, B_{8\rho}(x_{0});m)}{\rho^{\,n-1}} \right)\right\} \right|
\leqslant \\ \leqslant \left(1-\vartheta\,\left[\Lambda\left(x_{0}, \frac{\rho}{4}\right)\right]^{-\bar{\beta}_{3}} \right) |K_{\rho/4,\, 2\rho}|,
\end{multline}
where $K_{\rho_{1},\,\rho_{2}}:=B_{\rho_{2}}(x_{0})\setminus B_{\rho_{1}}(x_{0})$, \quad $0 < \rho_{1} < \rho_{2}$.
\end{lemma}
\begin{proof}
Let $\zeta_{1}\in C_{0}^{\infty}(B_{\rho}(x_{0}))$, $0\leqslant \zeta_{1}\leqslant1$,
$\zeta_{1}=1$ in $B_{\rho/2}(x_{0})$, $|\nabla \zeta_{1}|\leqslant 2/\rho$. Testing
\eqref{eq3.1} by $\varphi=v-m\,\zeta_{1}^{\,q}$ and using condition ($g_{\lambda}$) with $K=2\,M$
and the Young inequality \eqref{eq2.1},
we obtain for any $\varepsilon_{1}\in (\rho, 2\,M\,\lambda(\rho))$
$$
\begin{aligned}
\int\limits_{\mathcal{D}} H (x, |\nabla v|)\,dx
&\leqslant
\frac{\gamma m}{\rho} \int\limits_{K_{\rho/2,\, \rho}}
h(x, |\nabla v|)\,\zeta_{1}^{\,q-1} \,dx\leqslant
\\
&\leqslant\frac{\gamma m}{\varepsilon_{1}}\int\limits_{K_{\rho/2,\, \rho}}
H(x, |\nabla v|)\,dx+\frac{\gamma m}{\rho}\int\limits_{K_{\rho/2,\, \rho}}
h(x,\varepsilon_{1}/\rho)\,dx\leqslant
\\
&\leqslant \frac{\gamma m}{\varepsilon_{1}}\int\limits_{K_{\rho/2,\, \rho}} H (x, |\nabla v|)\,dx+\frac{\gamma m}{\rho}\int\limits_{K_{\rho/2,\, \rho}} b(x) g(x,\varepsilon_{1}/\rho)\,dx\leqslant
\\
&\leqslant \frac{\gamma m}{\varepsilon_{1}}\int\limits_{K_{\rho/2,\, \rho}} H (x, |\nabla v|)\,dx + \gamma m \Lambda\left(x_{0}, \frac{\rho}{4}\right)\,g(x_{0}, \varepsilon_{1}/\rho)\,\rho^{n-1}.
\end{aligned}
$$
Let $\zeta_{2}\in C_{0}^{\infty}(K_{\rho/4,\, 2\rho})$, $0\leqslant \zeta_{2}\leqslant1$,
$\zeta_{2}=1$ in $K_{\rho/2,\, \rho}$, $|\nabla \zeta_{2}|\leqslant 4/\rho$.
Testing \eqref{eq3.1} by $\varphi=v\,\zeta_{2}^{\,q}$ and using the Young inequality
\eqref{eq2.1},
we estimate the first term on the right-hand side of the previous inequality as follows:
\begin{multline*}
\int\limits_{K_{\rho/2,\, \rho}}\,H (x, |\nabla v|)\,dx\leqslant
\int\limits_{K_{\rho/4,\, 2\rho}}\,H (x, |\nabla v|)\,\zeta_{2}^{\,q}\,dx\leqslant
\gamma \int\limits_{K_{\rho/4,\, 2\rho}}\,H (x, v/\rho)\,\zeta_{2}^{\,q}\,dx\leqslant \\
\leqslant \gamma \int\limits_{K_{\rho/4,\, 2\rho}}\,b(x)\,G(x, v/\rho)\,dx \leqslant
\gamma \int\limits_{K_{\rho/4,\, 2\rho}}b(x)\,G^{+}_{K_{\rho/4,\, 2\rho}}(v/\rho)\,dx,
\end{multline*}
here we  used the notation $G^{+}_{K_{\rho/4,2\rho}}(v/\rho)=\sup\limits_{x \in K_{\rho/4,2\rho}}G(x, v/\rho)$.
Combining the last two inequalities and using the definition of capacity,
we obtain
\begin{multline*}
C_{H}(E, B_{8\rho}(x_{0});m)\leqslant
\frac{1}{m}\int\limits_{\mathcal{D}}H (x, |\nabla v|)\,dx
\leqslant
\frac{\gamma}{\varepsilon_{1}}\int\limits_{K_{\rho/4,2\rho}}b(x)
G^{+}_{K_{\rho/4,\, 2\rho}}( v/\rho)\,dx+\\ +\gamma\,\Lambda\left(x_{0}, \frac{\rho}{4}\right)\, g(x_{0},\varepsilon_{1}/\rho)\rho^{n-1}.
\end{multline*}
Choose $\varepsilon_{1}$ by the condition
$$
\Lambda\left(x_{0}, \frac{\rho}{4}\right)\,g(x_{0},\varepsilon_{1}/\rho)=\overline{\varepsilon}_{1}\,
\rho^{1-n}\,C_{H}(E, B_{8\rho}(x_{0});m),
\quad \overline{\varepsilon}_{1}\in (0,1).
$$
By \eqref{eq3.2} $\varepsilon_{1}\geqslant 8\rho$ if
$\bar{c}=\overline{c}(\overline{\varepsilon}_{1})$
is large enough and if $\bar{c}_{1} \geqslant 1$. To apply condition  {\rm (${\rm g}_{\lambda}$)}, we still need to check
the inequality  $\varepsilon_{1}\leqslant 2\,M\,\lambda(\rho)$.
Indeed, let $\psi\in C_{0}^{\infty}(B_{2\rho}(x_{0}))$, $0\leqslant \psi\leqslant1$,
$\psi=1$ in $B_{\rho}(x_{0})$ and $|\nabla \psi|\leqslant 2/\rho$, then by {\rm ($ g_{1}$)} and {\rm ($ g_{\lambda}$)}
with $K=2\,M$ we have
\begin{multline*}
C_{H}(E, B_{8\rho}(x_{0});m)\leqslant C_{H}(B_{\rho}(x_{0}), B_{8\rho}(x_{0});m)
\\
\leqslant
\frac{1}{m}\,\int\limits_{B_{8\rho}(x_{0})}
H(x, m|\nabla \psi|)\,dx
\leqslant
\frac{\gamma}{\rho} \int\limits_{B_{2\rho}(x_{0})}b(x)\,g(x, m/\rho)\,dx
\leqslant \gamma\,\Lambda\left(x_{0}, \frac{\rho}{4}\right)\,g(x_{0},m/\rho)\,\rho^{n-1},
\end{multline*}
and therefore, if $\overline{\varepsilon}_{1}=1/2\gamma$, then
$\varepsilon_{1}\leqslant \rho\, g^{-1}_{x_{0}}
\big( \overline{\varepsilon}_{1} \gamma\, g(x_{0}, m/\rho) \big)
\leqslant m \leqslant 2\,M\,\lambda(\rho)$.
So, by our choices, from the previous, we obtain
\begin{equation}\label{eq3.7}
C_{H} (E, B_{8\rho}(x_{0});m)
\leqslant \frac{\gamma}{\varepsilon_{1}}
\int\limits_{K_{\rho/4, 2\rho}} b(x)\,G^{+}_{K_{\rho/4,2\rho}}(v/\rho)\,dx.
\end{equation}
Let us estimate the term on the right-hand side of \eqref{eq3.7}.
For this we decompose $K_{\rho/4, 2\rho}$ as
$K_{\rho/4, 2\rho}=K'_{\rho/4, 2\rho}\cup K''_{\rho/4, 2\rho}$, where
$$
K'_{\rho/4, 2\rho}:=K_{\rho/4, 2\rho}\cap
\left\{ v\leqslant \varepsilon \, g^{-1}_{x_{0}}
\left( \frac{C_{H}(E, B_{8\rho}(x_{0});m)}{\rho^{\,n-1}\,[\Lambda\left(x_{0}, \frac{\rho}{4}\right)]^{\bar{\beta}_{4}}} \right)
 \right\},
 K''_{\rho/4, 2\rho}:=K_{\rho/4, 2\rho}\setminus K'_{\rho/4, 2\rho},
$$
and $\varepsilon\in (0,1), \bar{\beta}_{4} >0 $ to be determined  later.
If $\bar{c}^{\frac{1}{q-1}} \geqslant \dfrac{1}{\varepsilon}$ and $\bar{c}_{1} \geqslant \bar{\beta}_{4}$ then by \eqref{eq3.2}
$$\rho\leqslant \varepsilon \rho g^{-1}_{x_{0}}
\left( \frac{C_{H}(E, B_{8\rho}(x_{0});m)}{\rho^{\,n-1}[\Lambda\big(x_{0}, \frac{\rho}{4}\big)]^{\bar{\beta}_{4}}} \right) \leqslant 1 ,$$
and by {\rm (${\rm g}_{\lambda}$)} with $K=1$ we have
\begin{multline}\label{eq3.8}
\frac{\gamma}{\varepsilon_{1}}\int\limits_{K'_{\rho/4, 2\rho}}
b(x) G^{+}_{K_{\rho/4, 2\rho}}(v/\rho)\,dx \leqslant\\
\leqslant \frac{\gamma}{\varepsilon_{1}}\,\Lambda\left(x_{0}, \frac{\rho}{4}\right)\,G^{+}_{K_{\rho/4, 2\rho}}\bigg(\varepsilon  g^{-1}_{x_{0}}
\left( \frac{C_{H}(E, B_{8\rho}(x_{0});m)}{\rho^{\,n-1}[\Lambda\big(x_{0}, \frac{\rho}{4}\big)]^{\bar{\beta}_{4}}} \right)\bigg) \rho^{n}\leqslant \qquad\qquad\qquad\\
\leqslant \frac{\gamma}{\varepsilon_{1}}\Lambda\left(x_{0},\frac{\rho}{4}\right)G\left(x_{0}, \varepsilon g^{-1}_{x_{0}}
\left( \frac{C_{H}(E, B_{8\rho}(x_{0});m)}{\rho^{\,n-1}\left[\Lambda\left(x_{0}, \frac{\rho}{4}\right)\right]^{\bar{\beta}_{4}}} \right)\right)
 \leqslant \\ \leqslant \varepsilon \gamma\,\left[\Lambda\left(x_{0},\frac{\rho}{4}\right)\right]^{\frac{p}{p-1}-p\bar{\beta}_{4}}\,
C_{H}(E, B_{8\rho}(x_{0});m). \qquad\qquad\qquad\qquad
\end{multline}
By the upper bound for the function $v$, \eqref{eq3.2}
and our choice of $\varepsilon_{1}$, we obtain
\begin{equation}\label{eq3.9}
\frac{\gamma}{\varepsilon_{1}}\int\limits_{K''_{\rho/4, 2\rho}}
b(x)\,G^{+}_{K_{\rho/4, 2\rho}}(v/\rho)\,dx
\leqslant
\gamma(\varepsilon)\,\left[\Lambda\left(x_{0},\frac{\rho}{4}\right)\right]^{1+q\bar{\beta}}\,\rho^{-n}\,C_{H}(E, B_{8\rho}(x_{0});m)\,
%\left(1+ \frac{\rho}{G\left(x_{0}, g^{-1}_{x_{0}}
%\left( \frac{C(E, B_{8\rho}(x_{0});m)}{\rho^{\,n-1}} \right) \right)}  \right)
|K''_{\rho/4, 2\rho}|.
\end{equation}
Collecting estimates \eqref{eq3.7} -- \eqref{eq3.9} we obtain
\begin{multline*}
C_{H}(E, B_{8\rho}(x_{0});m) \leqslant \varepsilon \gamma\,\left[\Lambda\left(x_{0}, \frac{\rho}{4}\right)\right]^{\frac{p}{p-1}-p\bar{\beta}_{4}}\,
C_{H}(E, B_{8\rho}(x_{0});m)+\\ +\gamma(\varepsilon)\,\left[\Lambda\left(x_{0},\frac{\rho}{4}\right)\right]^{1+q\bar{\beta}}\,\rho^{-n}C{H}(E, B_{8\rho}(x_{0});m)\,|K''_{\rho/4, 2\rho}|,
\end{multline*}
choosing $\varepsilon$, $\beta$ from the conditions
$\varepsilon\gamma=1/2$, $\bar{\beta}_{4}= \frac{1}{p-1}$,
% and assuming that \eqref{eq3.6} is violated,
 we arrive at
\begin{multline*}
\left|\left\{K_{\rho/4, 2\rho} : v\geqslant \varepsilon \rho\, \left[\Lambda\left(x_{0},\frac{\rho}{4}\right)\right]^{-\frac{1}{(p-1)^{2}}}g^{-1}_{x_{0}}
\left( \frac{C_{H}(E, B_{8\rho}(x_{0});m)}{\rho^{\,n-1}} \right)
 \right\}\right| \geqslant \\ \geqslant
\left|\left\{K_{\rho/4, 2\rho} : v\geqslant \varepsilon \rho\, g^{-1}_{x_{0}}
\left( \frac{C_{H}(E, B_{8\rho}(x_{0});m)}{\rho^{\,n-1}[\Lambda\big(x_{0},\frac{\rho}{4}\big)]^{\frac{1}{p-1}}} \right)
 \right\}\right| \geqslant \gamma^{-1}\,\left[\Lambda\left(x_{0},\frac{\rho}{4}\right)\right]^{-1-q\bar{\beta}}\rho^{n}.
\end{multline*}
This completes the proof of the lemma with $\bar{\beta}_{2}= \frac{1}{(p-1)^{2}}$ and $\bar{\beta}_{3}= 1+ q\bar{\beta}$.
\end{proof}

\subsection{Proof of Theorem \ref{th1.1} under ($\mathbf{g_{\lambda}}$) conditions}

To prove Theorem \ref{th1.1} we need to estimate the term on the left hand side of inequality
\eqref{eq3.6}. Let $\delta=\dfrac{t}{t+1}$ and set $\overline{g}(x, v):= \frac{1}{v}\,\int\limits^{v}_{0} g(x, s)\,ds,$ $\overline{G}(x, v):= \int\limits^{v}_{0} \overline{g}(x, s)\,ds$ for $v > 0$. By our assumptions $ p-1\leqslant \dfrac{\overline{g}'(x, v) v}{\overline{g}(x, v)}\leqslant q-1,\quad x\in \Omega$.  We use Lemma \ref{lem2.3} with $Q(v)= \big[G^{-}_{B_{8\rho}(x_{0})}(v)]^{\delta}$, for this we need to check that $Q''(v) \geqslant 0$, indeed by our choices
\begin{multline*}
\frac{1}{\delta} Q''(v) [\overline{G}^{-}_{B_{8\rho}(x_{0})}(v)]^{2-\delta}= \overline{G}^{-}_{B_{8\rho}(x_{0})}(v) \overline{g}^{- '}_{B_{8\rho}(x_{0})}-\frac{1}{t+1}[\overline{g}^{-}_{B_{8\rho}(x_{0})}(v)]^{2}
 \geqslant \\ \geqslant
\frac{\overline{g}^{-}_{B_{8\rho}(x_{0})}(v)}{v} \left((p-1) \overline{G}^{-}_{B_{8\rho}(x_{0})}(v) - \frac{1}{t+1}\,v\overline{g}^{-}_{B_{8\rho}(x_{0})}(v)\right) \geqslant [\overline{g}^{-}_{B_{8\rho}(x_{0})}(v)]^{2}\left(\frac{p-1}{q}-\frac{1}{t+1}\right) >0.
\end{multline*}
Let $\varphi \in W_{0}(B_{8\rho}(x_{0}), \varphi \geqslant 1$ on $E$, then by Lemma \ref{lem2.3} and by the H\"{o}lder inequality we have
\begin{multline*}
\left[\overline{G}^{-}_{B_{8\rho}(x_{0})}\left(\frac{m}{\rho}\right)\right]^{\delta}\,|E| \leqslant
\int\limits_{B_{8\rho}(x_{0})}\left[\overline{G}^{-}_{B_{8\rho}(x_{0})}\left(\frac{m}{\rho}\varphi\right)\right]^{\delta}\,dx   \leqslant
\gamma\,\int\limits_{B_{8\rho}(x_{0})}\left[G^{-}_{B_{8\rho}(x_{0})}\left(m|\nabla \varphi|\right)\right]^{\delta}\,dx \leqslant \\
\qquad\qquad \leqslant
\gamma\,\left(\int\limits_{B_{8\rho}(x_{0})}[a(x)]^{-t}\,dx\right)^{1-\delta}\left(\int\limits_{B_{8\rho}(x_{0})}\,a(x)\,G\left(x, m|\nabla \varphi|\right) \,dx \right)^{\delta} \leqslant \\
\leqslant
\gamma\,\rho^{n(1-\delta)}\,[\Lambda(x_{0},\rho)]^{p\delta}\,\bigg(m\,C_{H}(E,B_{8\rho}(x_{0}); m)\bigg)^{\delta},
\end{multline*}
which yields by our choice of $m$ and $\delta$
\begin{equation*}
\gamma^{-1}\,g\left(x_{0}, \frac{m}{\rho}\right)\,\bigg(\frac{|E|}{\rho^{n}}\bigg)^{\frac{t+1}{t}}\,[\Lambda(x_{0}, \rho)]^{-p}\leqslant \rho^{1-n}\,C_{H}( E, B_{8\rho}(x_{0}); m),
\end{equation*}
and hence
\begin{equation}\label{eq3.10}
\rho\,g^{-1}_{x_{0}}\bigg(\rho^{1-n} C_{H}( E, B_{8\rho}(x_{0}); m)\bigg) \geqslant \gamma^{-1}\,m\,[\Lambda(x_{0},\rho)]^{-\bar{\beta}_{5}}\,\bigg(\frac{|E|}{\rho^{n}}\bigg)^{\frac{t+1}{t(p-1)}}, \quad \bar{\beta}_{5}=\frac{p}{p-1}.
\end{equation}
Note that by \eqref{eq3.10}, inequality \eqref{eq3.2} holds if
\begin{equation}\label{eq3.11}
m\,\left(\frac{|E|}{\rho^{n}}\right)^{\frac{t+1}{t(p-1)}}\geqslant \gamma\,\rho\,\left[\Lambda\left(x_{0},\frac{\rho}{4}\right)\right]^{\bar{\beta}_{6}},
\quad \bar{\beta}_{6}= \frac{\bar{c}_{1}}{p-1}+ \bar{\beta}_{5},
\end{equation}
with some $\gamma=\gamma(\bar{c})>0$.

In this case inequality \eqref{eq3.6} translates into
\begin{multline}\label{eq3.12}
\left| \left\{K_{\rho/4,\, 2\rho}:
v(x)\leqslant \gamma^{-1}\,\varepsilon\,m\, \left[\Lambda\left(x_{0}, \frac{\rho}{4}\right)\right]^{-\bar{\beta}_{2}-\bar{\beta}_{5}}\,\left(\frac{|E|}{\rho^{n}}\right)^{\frac{t+1}{t(p-1)}}  \right\} \right|
\leqslant \\ \leqslant \left(1-\vartheta\,\left[\Lambda\left(x_{0}, \frac{\rho}{4}\right)\right]^{-\bar{\beta}_{3}}\right)\, |K_{\rho/4,\, 2\rho}|.
\end{multline}
To complete the proof of Theorem \ref{th1.1} construct the sets $E(\rho,m):=B_{\rho}(x_{0})\cap \big\{u(x)\geqslant m \big\}$ and $\bar{E}(\rho, m):=B_{\rho}(x_{0})\cap \big\{u(x)\geqslant m\,\lambda(\rho) \big\}$,\quad $0< m < M$,\quad $E(\rho, m)\subset \bar{E}(\rho, m)$. Let $v$ be an auxiliary solution to the correspondent problem in  $\mathcal{D}:=B_{8\rho}(x_{0})\setminus \bar{E}(\rho, m)$. Since $u\geqslant v$ on $\partial\mathcal{D}$ then by the monotonicity condition  $u\geqslant v$ in $\mathcal{D}$ and inequality \eqref{eq3.12}  implies that
\begin{multline*}
\left| \left\{B_{2\rho}(x_{0}):
u(x)\leqslant \gamma^{-1}\,\varepsilon\,m\,\lambda(\rho)\left[\Lambda\left(x_{0}, \frac{\rho}{4}\right)\right]^{-\bar{\beta}_{2}-\bar{\beta}_{5}}\,
\bigg(\frac{|\bar{E}(\rho, m)|}{\rho^{n}}\bigg)^{\frac{t+1}{t(p-1)}}  \right\} \right|
\leqslant \\ \leqslant \left(1-\frac{1}{8}\vartheta\,\left[\Lambda\left(x_{0}, \frac{\rho}{4}\right)\right]^{-\bar{\beta}_{3}}\right)\, \, |B_{2\rho}(x_{0})|,
\end{multline*}
provided that inequality \eqref{eq3.11} holds.

Now we apply Lemma \ref{lem2.7} with  $r$ replaced by $\rho$, $m$ replaced by $m\,\lambda(\rho)$, $\xi \bar{\lambda}(r)\omega_{r}$ replaced by $\gamma^{-1}\varepsilon\,m\,\lambda(\rho)[\Lambda(x_{0}, \frac{\rho}{4})]^{-\bar{\beta}_{2}-\bar{\beta}_{5}} \,\bigg(\dfrac{|\bar{E}(\rho, m)|}{\rho^{n}}\bigg)^{\frac{t+1}{t(p-1)}} $ and $\alpha_{0}$ replaced by $\frac{1}{8}\vartheta\,[\Lambda(x_{0}, \frac{\rho}{4})]^{-\bar{\beta}_{3}}$. Lemma \ref{lem2.7} yields that if
\begin{equation}\label{eq3.13}
m\,\lambda(\rho)\,\bigg(\frac{|\bar{E}(\rho, m)|}{\rho^{n}}\bigg)^{\frac{t+1}{t(p-1)}}  \geqslant \gamma(\varepsilon)\,\rho\,\left[\Lambda\left(x_{0}, \frac{\rho}{4}\right)\right]^{\bar{\beta}_{2}+\bar{\beta}_{5}}
\exp\left(C_{*}\left[\Lambda\left(x_{0},\frac{\rho}{4}\right)\right]^{\bar{\beta}_{2}+\frac{\bar{\beta}_{3}}{\kappa_{1}}}\right),
\end{equation}
then
\begin{multline}\label{eq3.14}
\min\limits_{x\in B_{\frac{\rho}{2}}(x_{0})}u(x)\geqslant \gamma^{-1}\,\varepsilon\,m\,\lambda(\rho)\,\bigg(\frac{|\bar{E}(\rho, m)|}{\rho^{n}}\bigg)^{\frac{t+1}{t(p-1)}}\times\\
 \times\left[\Lambda\left(x_{0}, \frac{\rho}{4}\right)\right]^{-\bar{\beta}_{2}-\bar{\beta}_{5}}
 \exp\left(-C_{*}\left[\Lambda\left(x_{0},\frac{\rho}{4}\right)\right]^{\bar{\beta}_{2}+\frac{\bar{\beta}_{3}}{\kappa_{1}}}\right),
\end{multline}
here $\kappa_{1} >0$ is the number fixed by  Lemma \ref{lem2.7}. We can assume without loss that\\ $[\Lambda(x_{0}, \frac{\rho}{4})]^{\bar{\beta}_{2}+\bar{\beta}_{5}}\leqslant \exp\left(\Lambda\left(x_{0}, \frac{\rho}{4}\right)\right)$, therefore since $E(\rho, m) \subset \bar{E}(\rho, m)$, inequalities \eqref{eq3.13}, \eqref{eq3.14} imply \eqref{eq1.5} with $\beta_{1}=1+\bar{\beta}_{2}+\frac{\bar{\beta}_{3}}{\kappa_{1}}$, which completes the proof of Theorem \ref{th1.1}.

\section{Proof of Theorem \ref{th1.2} under ($\mathbf{g_{\mu}}$) conditions}\label{Sec4}
\medskip The proof is similar to that of Paragraphs $3.1-3.3$, we give only the sketch of the proof.

Fix $x_{0}\in \Omega$, let $E\subset B_{r}(x_{0})\subset B_{\rho}(x_{0})\subset B_{R}(x_{0})\subset\Omega$ and consider
the solution $w$ of the following problem
\begin{equation*}
{\rm div} \left(H(x, |\nabla w|) \frac{\nabla w}{ |\nabla w|^{2}}\right)=0, \quad x\in \mathcal{D}=B_{8\rho}(x_{0})\setminus E, \quad
w- m \psi \in W_{0}(\mathcal{D}),
\end{equation*}
where $m\in [\rho, 2M]$ is some fixed number and $\psi \in W_{0}(\mathcal{D})$, $\psi=1$ on $E$.

 We will assume that the following
integral identity holds:
\begin{equation}\label{eq4.1}
\int\limits_{\mathcal{D}} H(x, |\nabla w|)\,\frac{\nabla w}{|\nabla w|^{2}} \nabla\varphi \,dx=0
\quad \text{for any } \ \varphi\in W_{0}(\mathcal{D}),\quad \mathcal{D}=B_{8\rho}(x_{0})\setminus E.
\end{equation}

Testing \eqref{eq4.1} by $\varphi=(w-m)_{+}$ and by $\varphi=w_{-}$  and using condition
\eqref{eq1.4}, we obtain that $0\leqslant w\leqslant m$.
For $i, j=0,1,2, \ldots$ set $k_{j}:= k(1-2^{-j})$, $k>0$ to be chosen later,
$\rho_{i,j}:=2^{-i-j-3}\rho$,
$$
M_{i}:=\esssup\limits_{F_{i}}w, \quad
F_{i}:=\left\{ x\in \mathcal{D}: \frac{\rho}{4}(1+2^{-i})\leqslant |x-x_{0}|
\leqslant \frac{\rho}{2}(3-2^{-i}) \right\}.
$$
Fix $\overline{x}\in F_{i}$ and suppose that $(w(\overline{x})-k)_{+}\geqslant\rho$,
then $M_{i,j}(k_{j}):=\esssup\limits_{B_{\rho_{i,j}}(\overline{x})}\,(w-k_{j})\geqslant (w(\overline{x})-k)_{+}\geqslant\\ \geqslant \rho\geqslant \rho_{i,j}$.
And let
$\zeta_{i,j}\in C_{0}^{\infty}(B_{\rho_{i,j}}(\overline{x})), \ \
0\leqslant\zeta_{i,j}\leqslant 1, \ \
\zeta_{i,j}=1 \ \text{in} \ B_{\rho_{i,j+1}}(\overline{x}), \ \
|\nabla \zeta_{i,j}|\leqslant 2^{i+j+4}/\rho.$

By our choices of  we have
$w(x)\leqslant m\leqslant 2\,M$, $x\in B_{\rho_{i,j}}(\overline{x})$.
Therefore condition ($ g_{\mu}$) is applicable in $B_{\rho_{i,j}}(\overline{x})$ and we have
by ($ g_{\mu}$) with $K= M 2^{(i+j)\gamma}$ that
$$
g^{+}_{B_{\rho_{i,j}}(\overline{x})}\left( \frac{M_{i,j}(k_{j+1})}{\rho_{i,j}}\right)
\leqslant 2^{(i+j)\gamma}
g^{-}_{B_{\rho_{i,j}}(\overline{x})}\left(\frac{M_{i,j}(k_{j+1})}{\rho_{i,j}}\right).
$$
So, similarly to Section $3.1$ we obtain
\begin{multline*}
\int\limits_{B_{\rho_{i,j}}(\overline{x})}
|\nabla(w-k_{j+1})_{+}|\,\zeta_{i,j}^{\,q}\,dx
\leqslant\\ \leqslant
\gamma\,2^{\gamma(i+j)}\,M_{i+1}\,\mu\left(\frac{\rho}{4}\right)\,\Lambda\left(x_{0}, \frac{\rho}{4}\right)\,\rho^{n-1}
\,\left(\frac{|A^{+}_{\rho_{i,j}, k_{j+1}}|}{|B_{\rho_{i,j}}(\overline{x})|}\right)^{\left(1-\frac{1}{t(p-1)}\right)\left(1-\frac{1}{p}\right)+\left(1-\frac{1}{s}\right)\frac{1}{q}} \leqslant \\ \leqslant \gamma\,2^{\gamma(i+j)}
M_{i+1}\,\mu\left(\frac{\rho}{4}\right)\,\Lambda\left(x_{0}, \frac{\rho}{4}\right)\,\rho^{-n\kappa}\,k^{- 1- \kappa +\frac{1}{n}}\,\left(\int\limits_{B_{\rho_{i,j}}(\overline{x})} (v-k_{j})_{+}\,dx\right)^{1+\kappa -\frac{1}{n}},
\end{multline*}
where $\kappa= \dfrac{1}{n}-\dfrac{1}{tp}-\dfrac{1}{sq}-\dfrac{1}{p}+\dfrac{1}{q} > 0$.

From this, by standard arguments and
choosing $k$ from the condition
$$
k^{1+\frac{1}{\kappa}}= \gamma\,2^{i\gamma}\,M_{i+1}^{\frac{1}{\kappa}}\, \left[\mu\left(\frac{\rho}{4}\right)\,\Lambda(x_{0}, \frac{\rho}{4})\right]^{\frac{1}{\kappa}}\,\rho^{-n}\int\limits_{B_{\rho/2^{i+3}}(\overline{x})} w\,dx ,
$$
since $\overline{x}\in F_{i}$ is an arbitrary point, using the Young inequality and keeping in mind our assumption that $w(\overline{x})\geqslant k+\rho$ ,
from the previous we obtain for any $\varepsilon\in(0,1)$
\begin{equation}\label{eq4.2}
M_{i}\leqslant \varepsilon M_{i+1}+ \frac{\gamma\,2^{i\gamma}}{\varepsilon^{\gamma}\rho^{\,n}}\,
\left[\mu\left(\frac{\rho}{4}\right)\Lambda\left(x_{0}, \frac{\rho}{4}\right)\right]^{\frac{1}{\kappa}}\,
\int\limits_{F_{i+1}} w\,dx+\gamma\rho,
\quad i=0,1,2, \ldots
\end{equation}
Estimating the term on the right-hand side of inequality \eqref{eq4.2} completely similar to that of Section $3.1$ we arrive at
$$
g(x_{0},M_{0}/\rho)\leqslant \gamma\,\left[\mu\left(\frac{\rho}{4}\right)\Lambda\left(x_{0}, \frac{\rho}{4}\right)\right]^{1+\frac{q-1}{\kappa}}\,\rho^{1-n}\,
C_{H}(E, B_{8\rho}(x_{0}), m),
$$
which proves the upper bound of the function $w$.

The following lemma is analogous in formulation as well as in the
proof to Lemma \ref{lem3.2}.
\begin{lemma}\label{lem4.1}
There exist  numbers $\varepsilon$, $\vartheta\in(0,1)$, $\bar{\beta}_{2}, \bar{\beta}_{3} >0$ depending only on the data
such that
\begin{multline}\label{eq4.3}
\left| \left\{K_{\rho/4,\, 2\rho}:
w(x)\leqslant \varepsilon\,m\, \left[\mu\left(\frac{\rho}{4}\right)\,\Lambda\left(x_{0}, \frac{\rho}{4}\right)\right]^{-\bar{\beta}_{2}}\rho\, g^{-1}_{x_{0}} \left( \frac{C_{H}(E, B_{8\rho}(x_{0});m)}{\rho^{\,n-1}} \right)\right\} \right|
\leqslant \\ \leqslant \left(1-\vartheta\,\left[\mu\left(\frac{\rho}{4}\right)\,\Lambda\left(x_{0}, \frac{\rho}{4}\right)\right]^{-\bar{\beta}_{3}} \right) |K_{\rho/4,\, 2\rho}|.
\end{multline}
\end{lemma}
Similarly to \eqref{eq3.12}, inequality \eqref{eq4.3} translates into
\begin{multline*}
\left| \left\{K_{\rho/4,\, 2\rho}:
w(x)\leqslant \gamma^{-1}\,\varepsilon\,m\, \left[\mu\left(\frac{\rho}{4}\right)\,\Lambda\left(x_{0}, \frac{\rho}{4}\right)\right]^{-\bar{\beta}_{2}-\bar{\beta}_{5}}\,\bigg(\frac{|E|}{\rho^{n}}\bigg)^{\frac{t+1}{t(p-1)}}  \right\} \right|
\leqslant \\ \leqslant \left(1-\vartheta\,\left[\mu\left(\frac{\rho}{4}\right)\,\Lambda\left(x_{0}, \frac{\rho}{4}\right)\right]^{-\bar{\beta}_{3}}\right)\, |K_{\rho/4,\, 2\rho}|.
\end{multline*}
Set $E=E(\rho,N):=B_{\rho}(x_{0})\cap \big\{u(x)\geqslant m \big\}$
with any $0 < m \leqslant M$, by Lemma \ref{lem2.8}
 with  $r$ replaced by $\rho$, $\xi\,\omega_{r}$ replaced by $\gamma^{-1}\varepsilon\,m\,\left[\mu\big(\frac{\rho}{4}\big)\,\Lambda(x_{0}, \frac{\rho}{4})\right]^{-\bar{\beta}_{2}-\bar{\beta}_{5}} \,\bigg(\dfrac{|E(\rho, m)|}{\rho^{n}}\bigg)^{\frac{t+1}{t(p-1)}} $ and $\alpha_{0}$ replaced by $\frac{1}{8}\vartheta\,\left[\mu\big(\frac{\rho}{4}\big)\,\Lambda(x_{0}, \frac{\rho}{4})\right]^{-\bar{\beta}_{3}}$, from the previous inequality, using monotonicity condition  we obtain that either
\begin{equation}\label{eq4.4}
m\,\bigg(\frac{|E(\rho, m)|}{\rho^{n}}\bigg)^{\frac{t+1}{t(p-1)}}  \negthickspace\leqslant \negthickspace \gamma(\varepsilon)\,\rho\,\left[\mu\left(\frac{\rho}{4}\right)\negthickspace\Lambda\left(x_{0}, \frac{\rho}{4}\right)\right]^{\bar{\beta}_{2}+\bar{\beta}_{5}}\negthickspace\exp\left(C_{*}\left[\mu\left(\frac{\rho}{4}\right)\negthickspace\Lambda\left(x_{0},\frac{\rho}{4}\right)\right]
^{\bar{\beta}_{2}+\frac{\bar{\beta}_{3}}{\kappa_{1}}}\right),
\end{equation}
or
\begin{multline}\label{eq4.5}
\min\limits_{x\in B_{\frac{\rho}{2}}(x_{0})}\negthickspace \negthickspace \negthickspace u(x)\negthickspace\geqslant \negthickspace \gamma^{-1}\,\varepsilon\,m\,\bigg(\frac{|E(\rho, m)|}{\rho^{n}}\bigg)^{\frac{t+1}{t(p-1)}} \negthickspace \left[\Lambda\left(x_{0}, \frac{\rho}{4}\right)\right]^{-\bar{\beta}_{2}-\bar{\beta}_{5}}\negthickspace\negthickspace\negthickspace\exp\left(-C_{*}\left[\Lambda\left(x_{0},\frac{\rho}{4}\right)\right]
^{\bar{\beta}_{2}+\frac{\bar{\beta}_{3}}{\kappa_{1}}}\right),
\end{multline}
here $\kappa_{1} >0$ is the number fixed by Lemma \ref{lem2.7}. We  assume without loss that\\ $[\mu\big(\frac{\rho}{4}\big)\,\Lambda(x_{0}, \frac{\rho}{4})]^{\bar{\beta}_{2}+\bar{\beta}_{5}}\leqslant \exp\big(\mu\big(\frac{\rho}{4}\big)\,\Lambda(x_{0}, \frac{\rho}{4})\big)$, therefore  inequalities \eqref{eq4.4}, \eqref{eq4.5} imply \eqref{eq1.6} with $\beta_{1}=1+\bar{\beta}_{2}+\frac{\bar{\beta}_{3}}{\kappa_{1}}$, which completes the proof of Theorem \ref{th1.2}.

%%%%%%%%%%%%%%%%%%%%%%%%%%%%%%%%%%%%%%%%%%%%%%%%%%%%%%%%%%%%%%%%%%%%%%%%%%%%%%%%%%%%%%%%%%%%%%%%%%%%%%%%%%%%%%%%%
%%%%%%%%%%%%%%%%%%%%%%%%%%%%%%%%%%%%%%%%%%%%%%%%%%%%%%%%%%%%%%%%%%%%%%%%%%%%%%%%%%%%%%%%%%%%%%%%%%%%%%%%%%%%%%%%
%%%%%%%%%%%%%%%%%%%%%%%%%%%%%%%%%%%%%%%%%%%%%%%%%%%%%%%%%%%%%%%%%%%%%%%%%%%%%%%%%%%%%%%%%%%%%%%%%%%%%%%%%%%%%%%%%%

\section{ Harnack type inequalities,
proof of Theorems \ref{th1.3}, \ref{th1.4}}\label{Sec5}
\subsection{Proof of Theorem \ref{th1.3}}
Theorem \ref{th1.3} follows immediately from Theorems \ref{th1.1}, \ref{th1.2}, indeed set
$$\bar{m}(\rho)=\frac{1}{\lambda(\rho)}\exp\left(\gamma\,\left[\Lambda\left(x_{0},\frac{\rho}{4}\right)\right]^{\beta_{1}}\right)
\left(\min\limits_{B_{\frac{\rho}{2}}(x_{0})}u(x) + \rho \right),$$
where $\gamma$ and $\beta_{1}$ are the positive numbers defined in Theorem \ref{th1.1}. By \eqref{eq1.5}  we obtain for
$\theta\in\big(0,\frac{t}{t+1}(p-1)\big)$
\begin{multline}\label{eq5.1}
\rho^{-n}\,\int\limits_{B_{\rho}(x_{0})}u^{\theta}\,dx=\theta \,\rho^{-n}\,\int\limits_{0}^{\infty}|\{B_{\rho}(x_{0}): u(x) >m \}|\,m^{\theta-1}\,dm \leqslant \bar{m}^{\theta}(\rho)+\\+\gamma \bar{m}^{\frac{t}{t+1}(p-1)}(\rho)\int\limits_{\bar{m}(\rho)}^{\infty} m^{\theta-\frac{t}{t+1}(p-1)-1} \,dm \leqslant \gamma\,\big(t(p-1)-\theta(t+1)\big)^{-1}\,\bar{m}^{\theta}(\rho),
\end{multline}
which proves inequality \eqref{eq1.7}.

To prove  \eqref{eq1.8} we use inequality \eqref{eq1.6} with
$\bar{m}(\rho)\negthickspace=\negthickspace\exp\left(\gamma\,\left[\mu\left(\frac{\rho}{4}\right)\negthickspace\Lambda\left(x_{0},\frac{\rho}{4}\right)\right]
^{\beta_{1}}\right)\negthickspace\left(\min\limits_{B_{\frac{\rho}{2}}(x_{0})}\negthickspace u(x) + \rho \negthickspace\right)$,
completely similar to the previous
\begin{equation*}
\rho^{-n}\,\int\limits_{B_{\rho}(x_{0})}u^{\theta}\,dx\leqslant \bar{m}^{\theta}(\rho)+\gamma \bar{m}^{\frac{t}{t+1}(p-1)}(\rho)\int\limits_{\bar{m}(\rho)}^{\infty} m^{\theta-\frac{t}{t+1}(p-1)-1} \,dm \leqslant \gamma\big(t(p-1)-\theta(t+1)\big)^{-1} \bar{m}^{\theta}(\rho),
\end{equation*}
 which completes the proof of Theorem \ref{th1.3}.

\subsection{Proof of Theorem \ref{th1.4}}
The proof of Theorem \ref{th1.4} is almost standard. Fix $\sigma\in(0,\frac{1}{8})$, $s\in (\frac{3}{4}\rho, \frac{7}{8}\rho)$
and for $j=0,1,2,....$ set $\rho_{j}:=s (1-\sigma +\sigma  2^{-j})$, $B_{j}:=B_{\rho_{j}}(x_{0})$ and let $M_{0}:=\max\limits_{B_{0}} u$,
$M_{\sigma}:=\max\limits_{B_{\infty}} u$. Similarly to Section \ref{Sec4} we obtain
\begin{equation*}
M_{\sigma}^{1+\frac{1}{\kappa}} \leqslant \gamma\,\sigma^{-\gamma}\,M_{0}^{\frac{1}{\kappa}}\,
\left[\mu\left(\frac{\rho}{2}\right)\,\Lambda\left(x_{0}, \frac{\rho}{2}\right)\right]^{\frac{1}{\kappa}}\,\rho^{-n}\int\limits_{B_{0}}\,u\,dx +\gamma\,\rho,
\quad \kappa=\frac{1}{n}-\frac{1}{tp}-\frac{1}{sq}-\frac{1}{p}+\frac{1}{q} >0,
\end{equation*}
which implies for any $\varepsilon, \theta \in(0,1)$ that
\begin{equation*}
M_{\sigma} \leqslant \varepsilon\,M_{0} +\gamma\,\sigma^{-\gamma}\left[\mu\left(\frac{\rho}{2}\right)\,\Lambda\left(x_{0}, \frac{\rho}{2}\right)\right]^{\frac{1}{\kappa\theta}}\left(\rho^{-n}\int\limits_{B_{0}}\,u^{\theta}\,dx\right)^{\frac{1}{\theta}} +\gamma\,\rho.
\end{equation*}
Iterating this inequality we arrive at
\begin{equation}\label{eq5.2}
\sup\limits_{B_{\frac{\rho}{2}}(x_{0})} u \leqslant \gamma\,\left[\mu\left(\frac{\rho}{2}\right)\,\Lambda\left(x_{0}, \frac{\rho}{2}\right)\right]^{\frac{1}{\kappa\theta}}\left(\rho^{-n}\int\limits_{B_{0}}\,u^{\theta}\,dx\right)^{\frac{1}{\theta}} +\gamma\,\rho.
\end{equation}
Collecting inequalities \eqref{eq1.7} and \eqref{eq5.2} we obtain with $\theta=\frac{1}{2} \min\left(1,\frac{t}{t+1}(p-1)\right)$
\begin{multline*}
\sup\limits_{B_{\frac{\rho}{2}}(x_{0})} u \leqslant \frac{\gamma}{\lambda(\rho)}
\,\left[\mu\left(\frac{\rho}{2}\right)\,\Lambda\left(x_{0}, \frac{\rho}{2}\right)\right]^{\frac{1}{\kappa\theta}}\,
\exp\left(\gamma\,\left[\Lambda\left(x_{0},\frac{\rho}{4}\right)\right]^{\beta_{1}}\right)\left(\inf\limits_{B_{\frac{\rho}{2}}(x_{0})}u(x) + \rho \right)\leqslant \\
\leqslant \frac{\gamma}{\lambda(\rho)}\,\left[\mu\left(\frac{\rho}{4}\right)\right]^{\frac{1}{\kappa\theta}}\,
\exp\left(\gamma\,\left[\Lambda\left(x_{0},\frac{\rho}{4}\right)\right]^{\beta_{1}+ 1}\right)\left(\inf\limits_{B_{\frac{\rho}{2}}(x_{0})}u(x) + \rho \right),
\end{multline*}
which proves Theorem \ref{th1.4} under the $(g_{\lambda})$ and $(g_{\mu})$ conditions.

Similarly, by \eqref{eq1.8} and \eqref{eq5.2}  we obtain for any $\theta=\frac{1}{2}(1,\frac{t}{t+1}(p-1))$
\begin{equation*}
\sup\limits_{B_{\frac{\rho}{2}}(x_{0})} u \leqslant\gamma\exp\left(\gamma\,\left[\mu\left(\frac{\rho}{4}\right)\,\Lambda\left(x_{0},\frac{\rho}{4}\right)\right]^{\beta_{1}+\frac{1}{\kappa\theta}}\right)
\left(\inf\limits_{B_{\frac{\rho}{2}}(x_{0})}u(x) + \rho \right),
\end{equation*}
which completes the proof of Theorem \ref{th1.4} under $(g_{\mu})$ condition.

\subsection{Proof of Theorem \ref{th1.3} using the local clustering lemma under the \\conditions ($g_{\lambda}$) and $\Lambda(x_{0}, \rho)\asymp$ const }
In this Section, as it was mentioned, we give a simple proof of Theorem \ref{th1.3} for the function $u$ which belongs to the corresponding $DG_{g, 1}(\Omega)$ De Giorgi classes, using the local clustering lemma \cite{DiBGiaVes} under the conditions $(g_{\lambda})$ and  $\Lambda(x_{0}, \rho)\asymp$ const, $\rho > 0$. For the function $\lambda$ we additionally assume that there exists $c>0$ such that
\begin{equation*}
\lambda(\rho)\leqslant \bigg(\frac{\rho}{r}\bigg)^{c}\,\lambda(r),\quad 0 < r < \rho.
\end{equation*}
For $\rho >0$ denote by $K_{\rho}(y)$ a cube of edge $\rho$ centered at $y$.

\begin{lemma}{\rm{\textbf{(local clustering lemma \cite{DiBGiaVes})}}}
Let $u \in W^{1,1}(K_{\rho}(x_{0}))$  and let
\begin{equation*}
\int\limits_{K_{\rho}(x_{0})}|\nabla(u-k)_{-} |\,dx \leqslant \mathcal{K}\,k\,\rho^{n-1}\quad \text{and}\quad |\{K_{\rho}(x_{0}) : u > k \}|\geqslant \alpha\,|K_{\rho}(x_{0})|,
\end{equation*}
for some $k,\,\mathcal{K} >0$ and $\alpha \in(0,1)$. Then for any $\eta\in (0,1)$ and any $\delta \in(0,1)$ there exist $\bar{x} \in K_{\rho}(x_{0})$ and   $\bar{\varepsilon} =\bar{\varepsilon}(n) \in(0,1)$ such that
\begin{equation*}
|\{K_{r}(\bar{x}) : u \geqslant \delta\,k \}| \geqslant (1-\eta)\,|K_{r}(\bar{x})|,\quad r=\bar{\varepsilon}\,\eta\,(1-\delta)\,\frac{\alpha^{2}}{\mathcal{K}}\,\rho.
\end{equation*}
\end{lemma}

Let $\zeta\in C_{0}^{\infty}(K_{2\rho}(x_{0}))$, $\zeta=1$ on $K_{\rho}(x_{0})$, $0\leqslant \zeta\leqslant 1$, $|\nabla \zeta|\leqslant \dfrac{2}{\rho}$. For $m\in(0, M)$ set $k=m\lambda(r)$, where $r\in (0,\rho)$ to be chosen. If $ m \lambda(r)\geqslant \rho$, then by   condition
$(g_{\lambda})$
\begin{equation*}
\frac{g^{+}_{K_{2\rho}(x_{0})}\big(m\frac{\lambda(\bar{r})}{\rho}\big)}{g^{-}_{K_{2\rho}(x_{0})}\big(m\frac{\lambda(\bar{r})}{\rho}\big)}\leqslant C(M),
\end{equation*}
therefore inequality \eqref{eq2.4} can be rewritten as
\begin{equation*}
\int\limits_{K_{\rho}(x_{0})}\,|\nabla(u-k)_{-}|\,dx \leqslant \gamma\,k\,\rho^{n-1}.
\end{equation*}
Assume that with some $\alpha\in(0, 1)$
\begin{equation}\label{eq5.3}
|E(\rho, m)|:=|\{K_{\rho}(x_{0}): u\geqslant m\}| \geqslant \alpha\,|K_{\rho}(x_{0})|,
\end{equation}
then since
\begin{equation*}
|E(r, \rho, m)|:= |\{K_{\rho}(x_{0}): u\geqslant \lambda(r)\,m\}| \geqslant |E(\rho, m)| \geqslant \alpha\,|K_{\rho}(x_{0})|,
\end{equation*}
by the local clustering lemma with $k=\lambda(r)\,m$, $\delta=\frac{1}{2}$ and $\eta=\nu_{1}$  we obtain that
\begin{equation}\label{eq5.4}
|\{K_{r}(\bar{x}) : u \geqslant \frac{1}{2}\,\lambda(r)\,m \}| \geqslant (1-\nu_{1})\,|K_{r}(\bar{x})|,\quad r=\varepsilon\,\alpha^{2}\,\rho,\quad \varepsilon= \bar{\varepsilon}\,\,\frac{\nu_{1}}{2\,\gamma},
\end{equation}
here $\nu_{1}\in(0, 1)$ is the number fixed by  Lemma \ref{lem2.5}. Inequality \eqref{eq5.4} together with Lemma \ref{lem2.5} imply that
\begin{equation}\label{eq5.5}
\inf\limits_{K_{\frac{r}{2}}(\bar{x})}\,u \geqslant \frac{1}{8}\,\lambda(r)\,m ,
\end{equation}
provided that $\lambda(r)\,m \geqslant 8\,r$. By \eqref{eq5.5} and  Lemma \ref{lem2.7} there exists $C_{*} >0$ depending only on the data
such that
\begin{equation*}
\inf\limits_{K_{r}(\bar{x})}\,u \geqslant 2^{-C_{*}-3}\,\lambda(r)\,m ,
\end{equation*}
provided that $\lambda(r)\,m\geqslant 2^{C_{*}+3}\,r$. Repeating the previous arguments, on the $j-th$ step we obtain
\begin{equation*}
\inf\limits_{K_{2^{j} r}(\bar{x})}\,u \geqslant 2^{-j C_{*}-3}\,\lambda(r)\,m ,
\end{equation*}
provided that $\lambda(r)\,m\geqslant 2^{j C_{*}+3}\,2^{j}\,r$. Choosing $j$ from the condition $2^{j}\,r=\rho$, from the previous we obtain
\begin{equation*}
\inf\limits_{K_{\frac{\rho}{2}}(x_{0})}\,u \geqslant\inf\limits_{K_{2^{j} r}(\bar{x})}\,u \geqslant 2^{-j C_{*}-3}\,\lambda(r)\,m =
\frac{\varepsilon^{C_{*}}}{8}\,\alpha^{2C_{*}}\,\lambda(r)\,m\geqslant \frac{\varepsilon^{C_{*}+c}}{8}\,\alpha^{2(C_{*}+c)}\,\lambda(\rho)\,m,
\end{equation*}
provided that $\lambda(\rho)\,m\geqslant 8\,\varepsilon^{-C_{*}-c}\,\alpha^{-2(C_{*}+c)}\,\rho,$ which yield
\begin{equation}\label{eq5.6}
m\,\lambda(\rho)\,\alpha^{2(C_{*}+c)} \leqslant \gamma\,\left(\inf\limits_{K_{\frac{\rho}{2}}(x_{0})}\,u +\rho \right).
\end{equation}
This is an analog of inequality \eqref{eq1.5} of Theorem \ref{th1.1}. From this, using our choice of $\alpha$, similarly to inequality
\eqref{eq5.1} we obtain
for any $\theta \in\left(0, \frac{1}{2(C_{*}+c)}\right)$
\begin{equation}\label{eq5.7}
\left(\rho^{-n}\int\limits_{B_{\rho}(x_{0})}\,u^{\theta}\,dx\right)^{\frac{1}{\theta}}\leqslant \gamma\,(1-2(C_{*}+c)\theta)^{-\frac{1}{\theta}}\,[\lambda(\rho)]^{-1}\,\left(\inf\limits_{B_{\frac{\rho}{2}}(x_{0})}\,u +\rho \right),
\end{equation}
which completes the proof of Theorem \ref{th1.3} under the conditions $(g_{\lambda})$  and $\Lambda(x_{0}, \rho)\asymp$ const.

%\vskip 100pt

%%%%%%%%%%%%%%%%%%%%%%%%%%%%%%%%%%%%%%%%%%%%%%%%%%%%%%%%%%%%%%%%%%%%%%%%%%%%%%%%%%%%%%%%%%%%%%%%%%%%%%%%%%%%%%%%%
%%%%%%%%%%%%%%%%%%%%%%%%%%%%%%%%%%%%%%%%%%%%%%%%%%%%%%%%%%%%%%%%%%%%%%%%%%%%%%%%%%%%%%%%%%%%%%%%%%%%%%%%%%%%%%%%
%%%%%%%%%%%%%%%%%%%%%%%%%%%%%%%%%%%%%%%%%%%%%%%%%%%%%%%%%%%%%%%%%%%%%%%%%%%%%%%%%%%%%%%%%%%%%%%%%%%%%%%%%%%%%%%%%%

%\section{Section 5}\label{Sec5}

\vskip3.5mm
{\bf Acknowledgements.} This work is supported by grants of Ministry of Education and Science of Ukraine (project number is 0121U109525) and by the National Academy of Sciences of Ukraine (project numbers are 0121U111851, 0120U100178).

\bigskip

CONTACT INFORMATION

\medskip

Oleksandr V.~Hadzhy\\
Institute of Applied Mathematics and Mechanics,
National Academy of Sciences of Ukraine, Gen. Batiouk Str. 19, 84116 Sloviansk, Ukraine\\
aleksanderhadzhy@gmail.com

\medskip
Maria O. Savchenko\\
Institute of Applied Mathematics and Mechanics,
National Academy of Sciences of Ukraine, Gen. Batiouk Str. 19, 84116 Sloviansk, Ukraine\\
Vasyl' Stus Donetsk National University,
600-richcha Str. 21, 21021 Vinnytsia, Ukraine\\shan$\_$maria@ukr.net

\medskip
Igor I.~Skrypnik\\Institute of Applied Mathematics and Mechanics,
National Academy of Sciences of Ukraine, Gen. Batiouk Str. 19, 84116 Sloviansk, Ukraine\\
Vasyl' Stus Donetsk National University,
600-richcha Str. 21, 21021 Vinnytsia, Ukraine\\ihor.skrypnik@gmail.com

\medskip
Mykhailo V.~Voitovych\\Institute of Applied Mathematics and Mechanics,
National Academy of Sciences of Ukraine, Gen. Batiouk Str. 19, 84116 Sloviansk, Ukraine\\voitovichmv76@gmail.com

\end{document}